\documentclass[12pt, reqno]{amsart}
\usepackage{url,enumerate}
\usepackage{amsmath,amssymb,amsfonts,amstext, amsthm}
\usepackage{url,xspace}
\usepackage[hypertex]{hyperref}
\usepackage{verbatim}
\newtheorem{theorem}{Theorem}
\newtheorem{lemma}[theorem]{Lemma}
\newtheorem{proposition}[theorem]{Proposition}
\newtheorem{corollary}[theorem]{Corollary}
\newtheorem{example}{Example}
\newtheorem{remark}{Remark}
\newcommand{\egg}{\hfill \ensuremath{\Diamond}} 

\newtheorem{conjecture}{Conjecture}

\newtheorem{question}[conjecture]{Question}
\newcommand{\df}{\bf\boldmath}

\newcommand\supp[1]{  [#1] }

\newcommand\ns[1]{ \left\{ {#1} \right\} }

\newcommand\onee{{\mathbf{1}}}

\newcommand{\ind}{{\mathbf 1}}
\newcommand\one[1]{\ind_{#1}}
\newcommand{\eqd}{\stackrel{d}{=}}  

\renewcommand{\P}{{\mathbb P}}  
\newcommand{\Q}{{\mathbb Q}}
\newcommand{\Z}{{\mathbb Z}}
\newcommand{\R}{{\mathbb R}}
\newcommand{\N}{{\mathbb N}}
\newcommand{\E}{{\mathbb E}}     
\newcommand{\leb}{{\mathcal L}}  
\newcommand{\les}{{\mathcal L}}
\newcommand{\J}{{\mathcal J}}
\newcommand{\F}{{\mathcal F}}     

\newcommand{\A}{{\mathcal A}}     
\newcommand{\e}{\varepsilon}
\newcommand\gothh[1]{  {\mathcal #1} }
\newcommand\goth[1]{  {\mathfrak #1} }
\newcommand\fracc[2]{ #1 /#2 }      
\newcommand\Oo{ {\mathcal O } }
\newcommand\XX{{\mathbb M}}         
\newcommand\M{{\mathcal M}}        
\newcommand\onea[1]{ \mathbf{1}{ [#1]}    }
\newcommand\borelg{\goth{B}}
\newcommand\gothhh[1]{  { #1} }

\newcommand\ronnn{{|}}
\newcommand{\GZF}{\Upsilon_{\mathbb{C}}}
\newcommand{\GZH}{\Upsilon_{\mathbb{D}}}
\newcommand{\bGZH}{\widehat{\Upsilon}_{\mathbb{D}}}
\newcommand\nss[1]{ \left\{ {\ns{#1}} \right\} }
\DeclareMathOperator{\var}{Var}
\DeclareMathOperator{\cov}{Cov}
\newcommand{\erk}{\hfill \ensuremath{\Diamond}} 
\newcommand{\prob}{\xrightarrow{\P}}

\begin{document}

\title[Insertion and Deletion Tolerance of Point Processes]
{Insertion and Deletion Tolerance \\ of Point Processes}

\date{14 July 2010; revised 24 December 2012}

\author[A.\ E.\ Holroyd]{Alexander E.\ Holroyd}
\address[A.\ E.\ Holroyd]{Microsoft Research, 1 Microsoft Way,
Redmond, WA 98052, USA} \email{holroyd at microsoft.com}
\urladdr{http://research.microsoft.com/$\sim$holroyd/}

\author[T.\ Soo]{Terry Soo}
\address[T.\ Soo]{Department of Mathematics and Statistics, University of Victoria,
PO BOX 3060 STN CSC, Victoria, BC V8W 3R4, Canada} \email{tsoo at uvic.ca}
\urladdr{http://www.math.uvic.ca/$\sim$tsoo/}

\thanks{Funded in part by Microsoft Research (AEH) and NSERC (both authors)}
\subjclass[2010]{Primary 60G55} \keywords{point process, finite energy
condition, stable matching, continuum percolation}

\begin{abstract}
We develop a theory of insertion and deletion tolerance for point processes.
A process is insertion-tolerant if adding a suitably chosen random point
results in a point process that is absolutely continuous in law with respect
to the original process.  This condition and the related notion of
deletion-tolerance are extensions of the so-called finite energy condition
for discrete random processes.  We prove several equivalent formulations of
each condition, including versions involving Palm processes.  Certain other
seemingly natural variants of the conditions turn out not to be equivalent.
We illustrate the concepts in the context of a number of examples, including
Gaussian zero processes and randomly perturbed lattices, and we provide
applications to continuum percolation and stable matching.
\end{abstract}
\maketitle

\section{Introduction}

Let $\Pi$ be a point process on $\R^d$. Point processes will always be
assumed to be simple and locally finite. Let $\prec$ denote absolute
continuity in law; that is, for random variables $X$ and $Y$ taking values in
the same measurable space, $X \prec Y$ if and only if $\P(Y \in \A) = 0$
implies $\P(X \in \A) = 0$ for all measurable $\A$. Let $\borelg$ denote the
Borel $\sigma$-algebra on $\R^d$ and let $\les$ be Lebesgue measure. We say
that $\Pi$ is {\df{insertion-tolerant}} if for every $S \in \borelg$ with
$\les(S) \in (0, \infty)$, if $U$ is uniformly distributed on $S$ and
independent of $\Pi$, then
\begin{equation*}
\Pi + \delta_U \prec \Pi,
\end{equation*}
where $\delta_x$ denotes the point measure at $x \in \R^d$.

Let $\XX$ denote the space of simple point measures on $\R^d$. The support of
a measure $\mu \in \XX$ is denoted by
\begin{equation*}
[\mu]:= \ns{y \in \R^d : \mu(\ns{y}) =1}.
\end{equation*}
A {\df{$\boldsymbol{\Pi}$-point}} is an $\R^d$-valued random variable $Z$
such that $Z \in [\Pi]$ a.s.  A {\df{finite subprocess}} of $\Pi$ is a point
process $\gothh{F}$ such that $\gothh{F}(\R^d) < \infty$ and $[\gothh{F}]
\subseteq \supp{\Pi}$ a.s.    We say that $\Pi$ is {\df{deletion-tolerant}}
if for any  $\Pi$-point $Z$ we have
\begin{equation*}
\Pi -\delta_Z \prec \Pi.
\end{equation*}
For $S \in \borelg$ we define the restriction $\mu \ronnn_S$ of $\mu \in \XX$
to $S$ by
\begin{equation*}
\mu \ronnn_S(A) := \mu(A \cap S), \quad   \ A \in \borelg.
\end{equation*}

We will prove the following equivalences for insertion-tolerance and
dele\-tion-tolerance.

\begin{theorem}[Deletion-tolerance]
\label{equiv} 
Let $\Pi$ be a point process on $\R^d$.   The following are
equivalent.
\begin{enumerate}[(i)]
\item The point process $\Pi$ is
    deletion-tolerant.
\item
\label{minusF}
For any finite subprocess $\gothh{F}$ of $\Pi$, we
    have $\Pi -\gothh{F} \prec \Pi$.
\item 
\label{S} 
For all  bounded  $S \in \borelg$, we have
    $\Pi \ronnn_{S^c} \prec \Pi$.
    \item
    \label{conditionS}
    For all  bounded $S \in \borelg$,   we have  $\P\big(\Pi(S)= 0 \  \big| \  \Pi \ronnn_{S^c}\big) >0$ a.s.
\end{enumerate}
\end{theorem}

Condition \eqref{conditionS} has appeared previously under the name ``Condition ($\Sigma$).''   It appears to have been considered first in \cite{addthree,addtwo}; for further details see \cite{addone} and the references therein.



\begin{theorem}[Insertion-tolerance]
\label{thm-instol-eq}
Let $\Pi$ be a point process on $\R^d$. The following are equivalent.
\begin{enumerate}[(i)]
\item
The point process $\Pi$ is insertion-tolerant.
\item \label{finiteadd} For any Borel sets $S_1, \ldots, S_n$ of positive
    finite Lebesgue measure, if $U_i$ is a uniformly random point in
    $S_i$, with $U_1, \ldots, U_n$ and $\Pi$ all independent, then  $\Pi
    + \sum_{i=1} ^n \delta_{U_i} \prec \Pi$.
\item
\label{weakcond} If $(X_1, \ldots, X_n)$ is a random vector in $(\R^d)^n$
 that admits a conditional law given $\Pi$ that is absolutely continuous
 with respect to Lebesgue measure a.s., then    $\Pi + \sum_{i=1} ^n
 \delta_{X_i} \prec \Pi$.
\end{enumerate}
\end{theorem}
In fact we will prove a stronger variant of Theorem \ref{thm-instol-eq}, in which
(ii),(iii) are replaced with a condition involving the insertion of  a {\em random}
finite number of points.

We say that a point process is {\df{translation-invariant}} if it is
invariant in law under all translations of $\R^d$.  In this case further
equivalences are available as follows.

\begin{proposition}
\label{anyS} A translation-invariant point process $\Pi$ on $\R^d$ is
in\-ser\-tion\--toler\-ant if and only if there exists $S\in \borelg$ with $\les(S)
\in (0, \infty)$ such that, if $U$ is uniformly distributed in $S$ and
independent of $\Pi$, then $\Pi +\delta_U \prec \Pi$.
\end{proposition}

Let $\Pi$ be a translation-invariant point process with finite intensity;
that is,  $\E \Pi([0,1]^d) < \infty$.  We let $\Pi^*$ be its {\em Palm
version}.  See Section~\ref{palm} or \cite[Chapter 11]{MR1876169} for a
definition.  Informally, one can regard $\Pi^*$ as the point process $\Pi$
conditioned to have a point at the origin.

\begin{theorem}
\label{thm-instol-stat-eq} Let $\Pi$ be a translation-invariant ergodic point
process of finite intensity on $\R^d$  and let $\Pi ^{*}$ be its Palm
version. The following are equivalent.
\begin{enumerate}[(i)]
\item
The point process $\Pi$ is insertion-tolerant.
\item \label{original} $\Pi + \delta_0 \prec \Pi^*$.
\end{enumerate}
\end{theorem}
Condition \eqref{palmd} below appears to be the natural analogue of Theorem
\ref{thm-instol-stat-eq} \eqref{original} for deletion-tolerance.  However,
it is only sufficient and not necessary for deletion-tolerance.

\begin{theorem}
\label{suff} Let $\Pi$ be a translation-invariant point process of finite
intensity on $\R^d$  and let $\Pi^{*}$ be its Palm version.  If
\begin{equation}
\label{palmd}
\Pi^{*} -\delta_0 \prec \Pi,
\end{equation}
then $\Pi$ is deletion-tolerant.
\end{theorem}

In Section \ref{examples}, Example \ref{site} shows that a
dele\-tion-tolerant process need not satisfy \eqref{palmd}, while Example
\ref{counterS} shows that the natural analogue of Proposition~\ref{anyS}
fails for deletion-tolerance.   
Example \ref{site} will also show that boundedness of the set $S$ in Theorem \ref{equiv} (\ref{S}, \ref{conditionS})  cannot in general be replaced with the condition that it have finite volume.

\begin{remark}[More general spaces]
\label{gen} Invariant point processes and their Palm versions can
be defined on more general spaces than $\R^d$.  See
\cite{MR2371524,MR818219,MR2322698,newlast,lastjthp} for more information.
For concreteness and simplicity, we have chosen to state and prove Theorems
\ref{equiv}, \ref{thm-instol-eq}, \ref{thm-instol-stat-eq} and \ref{suff} in
the setting of $\R^d$, but they can easily be adapted to any complete
separable metric space endowed with: a group of symmetries that acts
transitively and continuously on it, and the associated Haar measure. We will
make use of this generality when we discuss Gaussian zero processes on the
hyperbolic plane in Proposition \ref{gausszeros}. \erk
\end{remark}
Next we will illustrate some applications of insertion-tolerance and
dele\-tion-tolerance in the contexts of continuum percolation and stable
matchings.  We will prove generalizations of earlier results.

The Boolean
continuum percolation model for point processes is defined as follows (see
\cite{roy}).  Let $\| \cdot\|$ denote the Euclidean norm on
$\R^d$.  For $R > 0$ and $\mu \in \XX$, consider the set
$\Oo(\mu):= \cup_{x \in \supp{\mu} } B(x,R)$, where $B(x,R):=
\ns{y \in \R^d: \|x-y\| < R}$ is the open ball of radius $R$
with center $x$.     We call $\Oo(\mu)$ the {\df{occupied
region}}.  The connected components of $\Oo(\mu)$ are called
{\df{clusters}}.

\begin{theorem}[Continuum percolation]
\label{percuniq}  Let $\Pi$ be a translation-in\-variant ergodic
insertion-tolerant point process on $\R^d$.  For any $R > 0$, the occupied
region $\Oo(\Pi)$ has at most one unbounded cluster a.s.
\end{theorem}
The proof of Theorem \ref{percuniq} is similar to the
uniqueness proofs in \cite[Chapter 7]{roy} which in turn are based on the argument of
Burton and Keane \cite{burtonkeane}.

Next we turn our attention to stable matchings of point processes (see
\cite{random} for background).   Let $\gothh{R}$ and $\gothh{B}$ be (`red'
and `blue') point processes on $\R^d$ with finite intensities.  A
{\df{one-colour matching scheme}} for $\gothh{R}$ is a point process
$\gothh{M}$ of unordered pairs $\ns{x,y} \subset \R^d$ such that almost
surely $[\gothh{M}]$ is the edge set of a simple graph  $([\gothh{R}],
[\gothh{M}])$ in which every vertex has degree exactly one.  Similarly, a
{\df{two-colour matching scheme}} for $\gothh{R}$ and $\gothh{B}$ is a point
process $\gothh{M}$ of unordered pairs $\ns{x,y} \subset \R^d$ such that
almost surely, $[\gothh{M}]$ is the edge set of a simple bipartite graph
$([\gothh{R}], [\gothh{B}], [\gothh{M}])$ in which every vertex has degree
exactly one.  In either case we write $\gothh{M}(x)=y$ if and only if
$\ns{x,y} \in [\gothh{M}]$.  In the one-colour case, we say that a matching
scheme is {\df{stable}} if almost surely there do not exist distinct points
$x,y \in [\gothh{R}]$ satisfying
\begin{equation}
\label{defstable}
\|x-y\| < \min\ns{ \|x - \gothh{M}(x)\|, \|y - \gothh{M}(y)\|},
\end{equation}
while in the two-colour case we say that a matching scheme is {\df{stable}}
if almost surely there do not exist $x \in [\gothh{R}]$ and $y \in
[\gothh{B}]$ satisfying \eqref{defstable}.  These definitions arise from the
concept of stable marriage as introduced by Gale and Shapley
\cite{galeshapley}.

It is proved in \cite{random} that stable matching schemes exist and are
unique for point processes that satisfy certain mild restrictions, as we
explain next. Let $\mu \in \XX$. We say that $\mu$ has a {\df{descending
chain}} if there exist $x_1, x_2, \ldots \in [\mu]$ with
$$\|x_{i-1} - x_i\| > \|x_i - x_{i+1}\| \ \text{for all} \ i.$$
 We say that $\mu$ is {\df{non-equidistant}} if for all $x,y,u,v
\in \supp{\mu}$ such that $\ns{x,y} \not = \ns{u,v}$ and  $x \not = y$ we
have $\|x-y\| \not = \|u-v\|$.  The following 
 facts are proved in
\cite[Proposition 9]{random}.  Suppose that $\gothh{R}$ is a
translation-invariant point process on $\R^d$ with finite intensity that
almost surely is non-equidistant and has no descending chains.  Then there
exists a one-colour stable matching scheme which is an isometry-equivariant
factor of $\gothh{R}$; this matching scheme may be constructed by a simple
procedure of iteratively matching, and removing, mutually-closest pairs of
$\gothh{R}$-points; furthermore, any two one-colour stable schemes agree
almost surely \cite[Proposition 9]{random}. In this case we refer to the
above-mentioned scheme as {\em the} one-colour stable matching scheme.
Similarly, in the two-colour case, let $\gothh{R}$ and $\gothh{B}$ be point
processes on $\R^d$ of equal finite intensity, jointly invariant and ergodic
under translations, and suppose that $\gothh{R} + \gothh{B}$ is a simple
point process that is almost surely non-equidistant and has no descending
chains.  Then there exists an almost surely unique two-colour stable matching
scheme, which is an isometry-equivariant factor and may be constructed by
iteratively matching mutually-closest $\gothh{R}$ / $\gothh{B}$ pairs.

Homogeneous Poisson process are non-equidistant and have no descending chains
(see \cite{haggstom-meester}).  Descending chains are investigated
in detail in \cite{jones}, where it is shown in particular that they are
absent in many well-studied point processes.

In this paper, our interest in stable matching lies in the typical distance
between matched pairs. Let $\gothh{M}$ be the one-colour stable matching
scheme for $\gothh{R}$. Consider the distribution function
\begin{equation}
\label{defdist}
F(r):= \big(\E\gothh{R}([0,1)^d) \big)^{-1}\E \#\ns{x \in [\gothh{R}] \cap [0,1)^d:
\|x - \gothh{M}(x) \| \leq r}.
\end{equation}
As in  \cite{random}, let $\gothhh{X}$ be a random variable with probability
measure $\P^{*}$ and expectation operator $\E^{*}$ such that $\P^{*}(X \leq
r) = F(r)$ for all $r \geq 0$. One may interpret $\gothhh{X}$ as the distance
from the origin to its partner under the Palm version of $(\gothh{R},
\gothh{M})$ in which we condition on the presence of an $\gothh{R}$-point at
the origin; see \cite{random} for details.  For the two-colour stable
matching scheme of point processes $\gothh{R},\gothh{B}$ we define $X$,
$\P^{*}$, and $\E^{*}$ in the same way.

\begin{theorem}[One-colour stable matching]
\label{onethm} Let $\gothh{R}$ be a translation-in\-variant
ergodic point process on $\R^d$ with finite intensity that
almost surely is non-equidistant and has no descending chains.
If $\gothh{R}$ is in\-sert\-ion-toler\-ant or deletion-tolerant, then
the one-colour stable matching sche\-me satisfies
 $\E^{*} \gothhh{X}^d = \infty$.
\end{theorem}

\begin{theorem}[Two-colour stable matching]
\label{twothm} \sloppypar Let $\gothh{R}$ and $\gothh{B}$ be independent
translation-invariant ergodic point processes on $\R^d$ with equal finite
intensity such that the point process $\gothh{R} +\gothh{B}$ is
non-equidistant and has no descending chains. If $\gothh{R}$ or $\gothh{B}$
is deletion-tolerant or insertion-tolerant, then the two-colour stable
matching scheme satisfies $\E^{*} \gothhh{X}^d = \infty$.
\end{theorem}

Theorems \ref{onethm} and \ref{twothm} strengthen the earlier results in
\cite{random} in the following ways. In \cite{random}, Theorem \ref{onethm}
is proved in the case of homogeneous Poisson processes, but the same proof is
valid under the condition that $\gothh{R}$ is {\em both} insertion-tolerant
{\em and} deletion-tolerant.  Similarly, in \cite{random},
Theorem~\ref{twothm} is proved in the Poisson case, but the proof applies
whenever $\gothh{R}$ or $\gothh{B}$ is insertion-tolerant.  Related results
appear also in \cite[Theorems 32,33]{Stable-PL}.

The following complementary bound is proved in \cite{random} for Poisson
processes, but again the proof given there applies more generally as follows.

\begin{theorem}[{\cite[Theorem 5]{random}}]
\label{thmfiverandom}  \sloppypar Let $\gothh{R}$ be a translation-invariant
ergodic non-equidistant point process on $\R^d$ with no descending chains,
and unit intensity.   The one-colour stable matching scheme
satisfies\linebreak $\P^{*}(X>r) \leq Cr^{-d}$ for all $r > 0$, for some
constant $C =C(d)$ that does not depend on $\gothh{R}$.
\end{theorem}

Thus, Theorems \ref{onethm} and \ref{thmfiverandom} provide strikingly close
upper and lower bounds on $X$ for the one-colour stable matching schemes of a
wide range of point processes.  For two-colour stable matching, even in the
case of two independent Poisson processes, the correct power law for the tail
of $X$ is unknown in dimensions $d\geq 2$; for $d=1$ the bounds 
 $\E^{*}
X^{\frac{1}{2}} = \infty$ and $\P^{*}(X > r) \leq Cr^{-\frac{1}{2}}$ hold. See
\cite{random} for details.

The rest of the paper is organized as follows.  In Section \ref{examples} we
present examples.  In Section \ref{easy} we prove some of the simpler results
including Theorems \ref{equiv} and \ref{thm-instol-eq}. Despite the
similarities between insertion-tolerance and deletion-tolerance, the proof of
Theorem \ref{thm-instol-eq} relies on the following natural lemma, whose
analogue for deletion-tolerance  is false (see Example \ref{nonmono}).	
\begin{lemma}[Monotonicity of insertion-tolerance]
\label{monofinite} Let $\Pi$ be a point process on $\R^d$ and let $S \in
\borelg$ have finite nonzero Lebesgue measure.   If $\Pi$ is
insertion-tolerant, and $U$ is uniformly distributed in $S$ and independent
of $\Pi$,  then $\Pi + \delta_U$ is insertion-tolerant.
\end{lemma}
Section \ref{palm} deals with Theorems \ref{thm-instol-stat-eq} and
\ref{suff}.  In Sections \ref{contperc} and \ref{stableM} we prove the
results concerning continuum percolation and stable matchings.  Section
\ref{perproof} provides proofs relating to some of the more elaborate
examples from Section \ref{examples}.

\section{Examples}
\label{examples}

First, we give examples of (translation-invariant) point processes that
possess various combinations of insertion-tolerance and deletion-tolerance.
We also provide examples to show that certain results concerning
insertion-tolerance do not have obvious analogues in the  setting of
deletion-tolerance.  Second, we give examples to show that the conditions in
the results concerning continuum percolation and stable matching are needed.
Finally, we provide results on perturbed lattice processes and Gaussian zeros
processes on the Euclidean and hyperbolic planes.

\subsection{Elementary examples}

\begin{example}[Poisson  process]
\label{poi} {\em{The homogeneous Poisson point process $\Pi$ on $\R^d$ is
both insertion-tolerant and deletion-tolerant. This follows immediately from
Theorem \ref{thm-instol-stat-eq} (\ref{original}) and Theorem \ref{suff} and
the relation $$\Pi ^{*} \eqd \Pi + \delta_0.$$ It is also easy to give an
direct proof of insertion-tolerance and to prove deletion-tolerance via
Theorem \ref{equiv} (\ref{S}). \egg}}
\end{example}

For $S \subseteq \R^d$ and  $x \in \R^d$, write $x + S :=\ns{x + z: z \in
S}$.
\begin{example}[Randomly shifted lattice]
\label{firsteg} {\em{Let $U$ be uniformly distrib\-uted in $[0,1]^d$.
Consider the point process given by $[\Lambda]:= U + \Z^d$. Clearly,
$\Lambda$ is translation-invariant.  Since no ball of radius $1/4$ can
contain more than one $\Lambda$-point,  by Theorem \ref{thm-instol-eq}
\eqref{finiteadd}, $\Lambda$ is not insertion-tolerant.  Also the cube
$[0,1]^d$ must contain $\Lambda$-points, so by Theorem \ref{equiv} \eqref{S},
$\Lambda$ is not deletion-tolerant.}} \egg
\end{example}
\begin{example}[Randomly shifted site percolation]
\label{site} {\em{Let $\ns{Y_z}_{z \in \Z^d}$ be i.i.d.\ \\$\ns{0,1}$-valued
random variables with $\E Y_0 =p \in(0,1)$.  Consider the random set
$W:=\ns{z \in \Z^d: Y_z=1}$.   Let $U$ be uniformly distributed in $[0,1]^d$
and independent of $W$. From Theorem \ref{equiv} \eqref{S}, it is easy to see
that $\Lambda$ given by $[\Lambda]:= U + W$ is deletion-tolerant.   Clearly,
as in Example \ref{firsteg}, $\Lambda$ is not insertion-tolerant.  Moreover,
it is easy to verify that almost surely $[\Lambda] \cap \Z^d = \emptyset$,
but $[\Lambda^*] \subset \Z^d$.  Thus (\ref{palmd}) is not satisfied.

 The {\em un}shifted process $\Phi$ given by $[\Phi]=W$ is also deletion-tolerant, but if $S$ is any set of finite volume containing the axis $\R\times\{0\}^{d-1}$, then $\Phi|_{S^{C}}\not\prec\Phi$.  Thus in Theorem 1 (\ref{S}, \ref{conditionS}) the boundedness condition on $S$ cannot be replaced with finite volume.  However, the randomly shifted process $\Lambda$ does satisfy $\Lambda|_{S^{C}}\prec\Lambda$ for $S$ as above.  See the discussion following the proof of Theorem \ref{equiv} in Section \ref{easy}.}}
\egg
\end{example}

\begin{example}
[Superposition of a Poisson point process with a randomly
shifted lattice] {\em{Let $\Pi$ be a Poisson point process on
$\R^d$ and let $\Lambda$ be a randomly shifted lattice (as in Example
\ref{firsteg}) that is independent of $\Pi$.  Consider the point
process $\Gamma:= \Pi + \Lambda$.  The
insertion-tolerance of $\Pi$ is inherited by $\Gamma$, but $\Gamma$ is
no longer deletion-tolerant.
As in Example \ref{firsteg}, $[0,1]^d$ must contain $\Gamma$-points.}} \egg
\end{example}

\begin{example}[Non-monotonicity of deletion-tolerance]
{\em{We show that in contrast with Lemma \ref{monofinite},
deleting a point from a deletion-tolerant process may destroy
deletion-tolerance. Let $(N_i)_{i\in\Z}$ be i.i.d., taking
values $0,1,2$ each with probability $1/3$, and let $\Pi$ have
exactly $N_i$ points in the interval $[i,i+1)$, for each
$i\in\Z$, with their locations chosen independently and
uniformly at random in the interval. It is easy to verify that
$\Pi$ is deletion-tolerant using Theorem \ref{equiv} \eqref{S}.

Consider the $\Pi$-point $Z$ defined as follows.  If the first integer
interval $[i, i+1)$ to the right of the origin that contains at least one
$\Pi$-point contains exactly two $\Pi$-points, then let $Z$ be the point in
this interval that is closest to the origin; otherwise, let $Z$ be the
closest $\Pi$-point to the left of the origin.  The point process $\Pi' = \Pi
-\delta_Z$ has the property that the first interval to the right of the
origin that contains any $\Pi$-points contains exactly one $\Pi$-point.

Let $Z'$ be the first $\Pi'$-point to the right of the origin.  The process
$\Pi ^{\prime \prime} :=\Pi' - \delta_{Z'}$ has the property that with
non-zero probability the first interval to the right of the origin that
contains any $\Pi ^{\prime \prime}$-points contains exactly two $\Pi ^{\prime
\prime}$-points.  Thus $\Pi'$ is not deletion-tolerant.

If desired, the above example can be made translation-invariant by applying a
random shift $U$ as before.  }}\egg \label{nonmono}
\end{example}

\begin{example}
[One set $S$ satisfying $\Pi \ronnn_{S^c} \prec \Pi$ does not suffice for
deletion-tolerance]
{\em  Let $\Lambda$ be a randomly shifted lattice in $d=1$ (as in Example
\ref{firsteg}) and let $\Pi$ be a Poisson point process on $\R$ of intensity
1 that is independent of $\Lambda$.  Let $Y:=\cup_{x \in [\Pi]} B(x, 5)$, and
consider $\Gamma:=\Lambda\ronnn_{Y^c}$.  Let $Z$ be the first $\Gamma$-point
to the right of the origin such that $Z + i \in [\Gamma]$ for all integers
$i$ with $|i| \leq 20$. Clearly, $\Gamma -\delta_Z \not \prec \Gamma$ and
thus $\Gamma$ is not deletion-tolerant.  On the other hand, since  $\Pi$ is
insertion-tolerant, $\Gamma\ronnn_{B(0,5)^c} \prec  \Gamma$. (Note the
contrast with Proposition \ref{anyS} for insertion-tolerance.) }\egg
\label{counterS}
\end{example}

\subsection{Continuum percolation and stable matching}

\begin{example}
[A point process that is neither insertion-tolerant nor deletion-tolerant and
has infinitely many unbounded clusters] {\em{Let $\ns{Y_z}_{z \in \Z}$ be
i.i.d.\ $\ns{0,1}$-valued random variables with $\E Y_0 = \frac{1}{2}$. Let
$$W:= \ns{(x_1, x_2) \in \Z^2: Y_{x_2} =1}$$ and let $U$ be uniformly
distributed in $[0,1]^2$ and independent of $W$.  Consider the point process
$\Lambda$ with support $U + W$.  Thus $\Lambda$ is a randomly shifted lattice
with columns randomly deleted.   As in Example~\ref{firsteg}, $\Lambda$ is
neither insertion-tolerant nor deletion-tolerant.  In the continuum
percolation model with parameter $R=2$, the occupied region $\Oo(\Lambda)$
has infinitely many unbounded clusters. \egg}}
\end{example}

\begin{example}
[A point process that is not insertion-tolerant, but is dele\-tion-toler\-ant
and has infinitely many  unbounded clusters] {\em{Let $\Lambda$ be a randomly
shifted super-critical site percolation in $d=2$, as in Example~\ref{site}.
Let $\ns{\Lambda_i}_{i \in \Z}$ be independent copies of $\Lambda$. Let
$\ns{Y_z}_{z \in \Z}$ be i.i.d.\ $\ns{0,1}$-valued random variables
independent of $\Lambda$ with $\E Y_0 = \frac{1}{2}$.   Consider the point
process $\Gamma$ with support $$[\Gamma]=\bigcup_{i \in \Z: \,Y_i=1}
[\Lambda_i] \times \ns{i}.$$  Thus $\Gamma$ is a point process in $\R^3$,
obtained by stacking independent copies of $\Lambda$.  Clearly, the point
process $\Gamma$ is deletion-tolerant, but not insertion-tolerant.  With
$R=2$, the occupied region $\Oo(\Gamma)$ has infinitely many unbounded
clusters.  \egg}}
\end{example}

\begin{example}[One-colour matching for two perturbed lattices]
{\em{Let $W=\ns{W_i}_{i \in \Z^d}$ and $Y=\ns{Y_i}_{i \in \Z^d}$ be all
i.i.d.\ random variables uniformly distributed in $B(0,1/4)$. Let $U$ be
uniformly distributed in $[0,1]^d$ and independent of $W,Y$. Let $\gothh{R}$
be the point process with support
$$[\gothh{R}] = U + \ns{i + W_i,i + Y_i: i \in \Z^d}.$$
It is easy to verify that $\gothh{R}$ is  neither insertion-tolerant nor
deletion-tolerant, and that $\gothh{R}$ has no descending chains and is
non-equidistant.  The one-colour stable matching scheme satisfies
$\|x-\gothh{M}(x)\| < \tfrac12$ for all $x \in [\gothh{R}]$ (in contrast with
the conclusion in Theorem \ref{onethm}). }} \egg \label{needone}
\end{example}

\begin{example}[Two-colour matching for randomly shifted lattices]
{\em Let $\gothh{R}$ and $\gothh{B}$ be two independent copies of the
randomly shifted lattice $\Z$ in $d=1$ as defined in Example \ref{firsteg}.
Although $\gothh{R} + \gothh{B}$ is not non-equidistant, it is easy to verify
that there is an a.s.\ unique two-colour stable matching scheme for
$\gothh{R}$ and $\gothh{B}$, and it satisfies $\|x -\gothh{M}(x)\| <
\tfrac{1}{2}$ for all $x \in [\gothh{R}]$.} \egg
\end{example}

\subsection{Perturbed lattices and Gaussian zeros}
The proofs of the results stated below are given in Section  \ref{perproof}.

\begin{example}[Perturbed lattices]
\label{pert} {\em{Let $\ns{Y_z}_{z \in \Z^d}$ be i.i.d.\ $\R^d$-valued random
variables.   Consider the point process $\Lambda$ given by $$[\Lambda]:=
\ns{z + Y_z: z \in \Z^d}.$$
Note that $\Lambda$ is invariant and ergodic under shifts of $\Z^d$. It is
easy to see that (for all dimensions $d$) if $Y_0$ has bounded support, then
$\Lambda$ is neither insertion-tolerant nor deletion-tolerant. Indeed, in
this case we have  $\Lambda(B(0,1)) \leq M$ for some constant $M<\infty$, so,
by Theorem \ref{thm-instol-eq} \eqref{finiteadd}, $\Lambda$ is not
insertion-tolerant (otherwise we could add $M+1$ random points in $B(0,1)$).
Also, $\Lambda(B(0,N)) \geq 1$, for some $N<\infty$, so Theorem \ref{equiv}
\eqref{S} shows that $\Lambda$ is not deletion-tolerant. \egg}}
\end{example}

For dimensions $1$ and $2$ we can say more.
\begin{proposition}[Perturbed lattices in dimensions $1,2$]
\label{pertone} Let $[\Lambda]:= \ns{z + Y_z: z \in \Z^d}$ for i.i.d.\
$\ns{Y_z}_{z \in \Z^d}$.  For $d=1,2$, if $\E \|Y_0\|^d < \infty$, then
$\Lambda$ is neither insertion-tolerant nor deletion-tolerant.
\end{proposition}

\begin{question}
Does there exists a distribution for the perturbation $Y_0$ such that the
resulting perturbed lattice  is insertion-tolerant?  In particular, in the
case $d=1$, does this hold whenever $Y_0$ has infinite mean?  What are the
possible combinations of insertion-tolerance and deletion-tolerance for
perturbed lattices?  Allan Sly has informed us that he has made progress on
these questions.
\end{question}

Perturbed lattice models were considered by Sodin and Tsirelson
\cite{MR2121537} as simplified models to illustrate certain properties of
Gaussian zero processes (which we will discuss next).  Our proof of
Proposition~\ref{pertone} is in part motivated by their remarks, and similar
proofs have also been suggested by Omer Angel and Yuval Peres (personal
communications).

The Gaussian zero processes on the plane and hyperbolic planes are defined as
follows (see \cite{MR2552864,MR2121537} for background).  Let $\ns{a_n}_{n=0}
^ {\infty}$ be i.i.d.\ standard complex Gaussian random variables with
probability density $\pi^{-1}\exp(-|z|^2)$ with respect to Lebesgue measure
on the complex plane.  Firstly, consider the entire function
\begin{equation}
\label{planef}
f(z) := \sum_{n=0} ^ {\infty} \frac{a_n}{\sqrt{n!}} z^n.
\end{equation}
The set of zeros of $f$ forms a translation-invariant point process $\GZF$ in
the complex plane.  Secondly, consider the analytic function on the unit disc
$\mathbb{D}:=\ns{z \in \mathbb{C} : |z| < 1}$ given by
\begin{equation}
\label{hypf}
g(z):= \sum_{n=0} ^ {\infty} a_n z^n.
\end{equation}
The set of zeros of $g$ forms a point process $\GZH$.  We endow $\mathbb{D}$
with the hyperbolic metric $|dz| / (1 - |z|^2)$ and the group of symmetries
$G$ given by the maps $z \mapsto (az + b) /(\bar{b}z + \bar{a})$, where $a,
b\in \mathbb{C}$ and  $|a|^2 - |b|^2 =1$.  Then $\GZH$ is invariant in law
under action of $G$.

The following two facts were suggested to us by Yuval Peres, and are
consequences of results of \cite{MR2121537} and \cite{MR2231337}
respectively.
\begin{proposition}
\label{GAFplane} The Gaussian zero process $\GZF$ on the plane is neither
insertion-tolerant nor deletion-tolerant.
\end{proposition}
\begin{proposition}
\label{gausszeros} The Gaussian zero process $\GZH$ on the hyperbolic plane
is both insertion-tolerant and deletion-tolerant.
\end{proposition}

\section{Basic results}
\label{easy}

In this section we prove elementary results concerning
insertion-tolerance and deletion-tolerance.  The following
simple application of Fubini's theorem will be useful. Recall
that $\leb$ denotes Lebesgue measure.
\begin{remark}
\label{fubini} Let $\Pi$ be a point process on $\R^d$.   If $S \in \borelg$
is a set of positive finite measure and $U$ is uniformly distributed $S$ and
independent of $\Pi$, then

\hfill $\displaystyle\P( \Pi + \delta_U \in \cdot) = \frac{1}{\les(S)}\int_S
\P(\Pi + \delta_x\in \cdot)\,dx.$ \hfill $\Diamond$
\end{remark}
Let $\goth{M}$ be the product $\sigma$-field on $\XX$. For $\A \in \goth{M}$
and $x\in\R^d$, we set $$\A^x := \{ \mu \in \XX : \mu + \delta_x \in \A \}.$$
Thus $\A^{x}$ is the set of point measures for which adding a point at $x$
results in an element of $\A$.

\begin{proof}[Proof of Lemma \ref{monofinite}]
Let $\Pi$ be insertion-tolerant. We first show that for almost all $x \in
\R^d$ the point process $\Pi + \delta_x$ is insertion-tolerant.  The proof
follows from the definition of $\A^x$.  Let $V$ be uniformly distributed in
$S' \in \borelg$ and independent of $\Pi$. Suppose $\A \in \goth{M}$ is such
that $$\P(\Pi + \delta_x + \delta_V \in \A) = \P(\Pi + \delta_V \in \A^x) >
0.$$ Since $\Pi$ is insertion-tolerant, $0<\P(\Pi \in \A^x)=\P(\Pi+ \delta_x
\in \A)$.

Next, let $U$ be uniformly distributed in $S \in \borelg$ and independent of
$(\Pi, V)$.   Let  $\P(\Pi + \delta_U \in \A)=0$, for some $\A \in \goth{M}$.
By Remark \ref{fubini},  $\P(\Pi + \delta_x \in \A) = 0$ for almost all $x
\in S$, and since $\Pi + \delta_x$ is insertion-tolerant for almost all $x
\in \R^d$, we deduce that $\P(\Pi + \delta_x + \delta_V \in \A) =0$ for
almost all $x \in S$.  Applying Remark \ref{fubini} to the process $\Pi +
\delta_V$, we obtain $\P(\Pi + \delta_U +\delta_V \in \A) =0$.
\end{proof}

With Lemma \ref{monofinite} we prove that insertion-tolerance implies the
following stronger variant of Theorem \ref{thm-instol-eq} in which we allow
the number of points added to be random. If $(X_1, \ldots, X_n)$ is a random
vector in $(\R^d)^n$ with law that is absolutely continuous with respect to
Lebesgue measure, then we say that the random (unordered) set $\ns{X_1,
\ldots, X_n}$ is {\df{nice}}.  A finite point process $\gothh{F}$ is
{\df{nice}} if for all $n \in \N$, conditional on $\gothh{F}(\R^d) = n$, the
support $[\gothh{F}]$ is equal in distribution to some nice random set; we
also say that the law of $\gothh{F}$ is nice if $\gothh{F}$ is nice.

\begin{corollary}
\label{weak} Let $\Pi$ be an in\-sert\-ion-toler\-ant point process on $\R^d$
and let $\gothh{F}$ be a  finite point process on $\R^d$.  If $\gothh{F}$
admits a conditional law given $\Pi$ that is nice, then $\Pi + \gothh{F}
\prec \Pi$.
\end{corollary}

\begin{proof}[Proof of Theorem \ref{thm-instol-eq}]
Clearly, (iii) $\Rightarrow$ (ii) $\Rightarrow$  (i).  From Corollary
\ref{weak}, it is immediate that (i) $\Rightarrow$   (iii).
\end{proof}

\begin{proof}[Proof of Corollary \ref{weak}]
Let $U$ be uniformly distributed in $[0,1]$ and independent of $\Pi$.  Let
$f:\XX \times [0,1] \to \XX$ be a measurable function such that for all $\pi
\in \XX$ we have that $f(\pi, U)$ is a nice finite point process.  It
suffices to show that  $\Pi + f(\Pi, U) \prec \Pi$.

Consider the events
$$E_{n,k}:= \Big\{ f(\Pi, U)(\R^d) = n \Big\} \
\cap \  \Big\{[f(\Pi, U)] \subset B(0,k)\Big\}.$$
Let  $\ns{U_{r,k}}_{r=1} ^n$ i.i.d.\ random variables uniformly distributed
in $B(0,k)$ and independent of $(\Pi, U)$.  Let $\gothh{F}'_{n,k}:=
\sum_{r=1} ^n \delta_{U_{r,k}}$.  By applying Lemma \ref{monofinite},  $n$
times, we see that $\Pi + \gothh{F}_{n,k}' \prec \Pi$; thus it suffices to
show  that $\Pi +f(\Pi, U) \prec \Pi +\gothh{F}_{n,k}'$ for some $n,k \geq
0$.

For each $\mathbf{x} \in (\R^d)^n$,  let $(\mathbf{x}_1, \ldots,
\mathbf{x}_n) = \mathbf{x}$.  If $S \subset \R^d$ has $n$ elements, then we
write $\langle S \rangle:= (s_1 \ldots, s_n) \in (\R^d)^n,$ where $s_i$ are
the elements of $S$ in lexicographic order.  For each $n \geq 0$, let
$g_n:(\R^d)^n \times \XX \to \R$ be a measurable function such that
$g_n(\cdot, \pi)$ is the probability density function (with respect to
$n$-dimensional Lebesgue measure) of $\langle [f(\pi, U)] \rangle $,
conditional on $f(\pi, U)(\R^d) =n$.  Let $Q$ be the law of $\Pi$ and let $\A
\in \goth{M}$.  Thus
\begin{eqnarray}
\label{nk}
\lefteqn{\P\big(\Pi + f(\Pi, U) \in \A, \ E_{n,k}\big)=} \nonumber\\
&& \int \bigg(
\int_{B(0,k)^n} \onee\Big[\pi+  \sum_{i=1} ^n {\delta_{\mathbf{x}_i}}  \
\in \A \Big]g(\mathbf{x}, \pi)d\mathbf{x}\bigg)dQ(\pi).
\end{eqnarray}
On the other hand,
\begin{eqnarray}
\label{nkprime}
\lefteqn{\P\big(\Pi + \gothh{F}_{n,k}'  \in \A\big) =} \nonumber \\ && \int
\frac{1}{\les(B(0,k))^n}\bigg( \int_{B(0,k)^n} \onee\Big[\pi + \sum_{i=1} ^n
{\delta_{\mathbf{x}_i}}  \ \in \A \Big] d\mathbf{x}\bigg)dQ(\pi).
\end{eqnarray}
If $\P(\Pi + f(\Pi, U) \in \A) > 0$, then there exist $n,k \geq 0$ such that
$\P(\Pi + f(\Pi, U) \in \A, \ E_{n,k}) >0$; moreover from \eqref{nk} and
\eqref{nkprime}, we deduce that $\P(\Pi + \gothh{F}_{n,k}'  \in \A) >0$.
\end{proof}

The proof of Theorem \ref{equiv} relies on the following lemma.
\begin{lemma}
\label{unipick} Let $\Pi$ be a point process on $\R^d$. If   $\gothh{F}$ is a
finite subprocess of $\Pi$, then  there exists a bounded $S \in \borelg$ with $\les(S)
\in (0, \infty)$ such that
\begin{equation}
\label{revS}
\P (\Pi\ronnn_{S} = \gothh{F} ) >0.
\end{equation}
\end{lemma}
\begin{proof}
 A ball $B(x,r)$ is {\df{rational}} if $x \in \Q^d$ and $r \in \Q^{+}$.
 Let $C$ be the collection of all unions of finitely many rational balls.
Clearly $C$ is countable.  We will show that there exists $S \in C$
satisfying \eqref{revS}.  Since $\Pi$ is locally finite, it follows that
there exists a $C$-valued random variable $\mathbf{S}$ such that
$\Pi\ronnn_{\mathbf{S}} = \gothh{F}$ a.s.   Since $$ \sum_{S \in C
}\P(\Pi\ronnn_{S} =\gothh{F}, \ \ \mathbf{S} =S) =\P(\Pi\ronnn_{\mathbf{S}}
=\gothh{F}) =1,$$ at least one of the terms of the sum is nonzero.
\end{proof}

With Lemma \ref{unipick} we  first prove the following special case of
Theorem~\ref{equiv}.
%
\begin{lemma}
\label{minieq} Let $\Pi$ be a point process on $\R^d$.  The following conditions are
equivalent.
\begin{enumerate}[(i)]
\item
The point process $\Pi$ is deletion-tolerant.
\item If $\gothh{F}$ is a finite subprocess of $\Pi$ such that
    $\gothh{F}(\R^d)$ is a bounded random variable, then $\Pi -\gothh{F}
    \prec \Pi$.
\end{enumerate}
\end{lemma}
\begin{proof}
Clearly, (ii) implies (i).

We show by induction on the number of points of the finite subprocess that
(i) implies (ii).  Assume that $\Pi$ is  deletion-tolerant. Suppose that (ii)
holds for every finite subprocess $\gothh{F}$ of  $\Pi$ such that
$\gothh{F}(\R^d) \leq n$.  Let $\gothh{F}'$ be a finite subprocess of $\Pi$
with $\gothh{F}'(\R^d) \leq n+1$.  Observe that on the event that
$\gothh{F}'(\R^d) \not = 0$, we have $\gothh{F}'= \gothh{F} +\delta_{Z},$
where $\gothh{F}$ is a finite subprocess of
 $\Pi$ with $\gothh{F}(\R^d) \leq n$ and $Z$ is some
 $\Pi$-point.  Let $\P( \Pi -\gothh{F}' \in \A
) > 0$, for some $\A \in \goth{M}$.  If $\P(\Pi -\gothh{F}' \in \A, \ \
\gothh{F}'(\R^d) = 0 ) > 0$, then clearly $\P(\Pi \in \A) > 0$.  Thus we
assume without loss of generality that $\gothh{F}'= \gothh{F}  +\delta_{Z}$
so that $\P(\Pi - \gothh{F} - \delta_{Z} \in \A) > 0$.
By applying Lemma~{\ref{unipick}} to the point process $\Pi - \gothh{F}$,
conditioned on $\Pi - \gothh{F} - \delta_Z \in \A$,   there exists $S \in
\borelg$ with finite Lebesgue measure, so that
\begin{equation}
\label{addAS}
\P\Bigl( (\Pi-\gothh{F})\ronnn_{S} = \delta_Z \
\Big| \ \Pi - \gothh{F} - \delta_Z \in \A \Bigr) > 0.
\end{equation}
Let $\A^S:= \ns{ \mu + \delta_x : \mu \in \A,\;  x \in S}$, so that by the
definition of $\A^S$ and \eqref{addAS}, we have $\P(\Pi -\gothh{F} \in \A^S)
> 0$.  By the inductive hypothesis,  $\P(\Pi \in \A^S) > 0$.

Observe that if $\Pi \in \A^S$, there is an $x \in \supp{\Pi} \cap S$ such
that $\Pi - \delta_x \in \A$.  Define a $\Pi$-point $R$ as follows.  If $\Pi
\in \A ^S$, let $R$ be the point of $\supp{\Pi} \cap S$ closest to the origin
(where ties are broken using lexicographic order) such that $\Pi -\delta_{R}
\in \A$, otherwise  let $R$ be the $\Pi$-point closest to the origin.
Hence $$\P(\Pi-\delta_R \in \A) \geq \P(\Pi \in \A ^S) > 0.$$  Since $\Pi$ is
deletion-tolerant, $\P(\Pi \in \A) >0$.
\end{proof}

\begin{proof}[Proof of Theorem {\ref{equiv}}]

We show that \eqref{conditionS}  $\Rightarrow$ \eqref{S} $\Rightarrow$  (i) $\Rightarrow$  \eqref{minusF}
$\Rightarrow$ \eqref{S} $\Rightarrow$ \eqref{conditionS}.

Assume that \eqref{conditionS}  holds and that for some bounded $S \in \borelg$  and some $A \in \goth{M}$ we have 
$\P( \Pi \ronnn_{S^c} \in A)>0$.    
From \eqref{conditionS}, we have
 $\P(\Pi(S) = 0 \ | \ \Pi\ronnn_{S^c}) > 0$ a.s.    
Thus $\P(\Pi \in A ) \geq  \P(\Pi \ronnn_{S^c} \in A, \Pi(S) = 0) >0$, and  \eqref{S}  holds.

Assume that \eqref{S}  holds and that for some $\Pi$-point
$Z$ and some $\A \in \goth{M}$ we have $\P( \Pi - \delta_Z \in \A) > 0$.  By Lemma
{\ref{unipick}}, $ \P(\Pi\ronnn_{S^c} \in \A) >
0$ for some bounded $S \in \borelg$, with finite Lebesgue measure.
From (\ref{S}), $\P(\Pi \in \A) > 0$.  Thus (i)
holds and $\Pi$ is deletion-tolerant.

Assume that (i) holds.  Let $\gothh{F}$ be a finite subprocess of $\Pi$ and
suppose for some $\A \in \goth{M}$ we have $\P(\Pi -\gothh{F} \in \A ) >0$.
Define $\gothh{F}_n$ as follows.  Take $\gothh{F}_n =\gothh{F}$ if
$\gothh{F}(\R^d)=n$, otherwise set $\gothh{F}_n = 0$.     Note that for some
$n$, we have $\P(\Pi - \gothh{F}_n \in \A) > 0$.  Since $\Pi$ is
deletion-tolerant, by Lemma~{\ref{minieq}}, $\P(\Pi \in \A) > 0$.  Thus
(\ref{minusF}) holds.

Clearly (\ref{minusF}) implies (\ref{S}), since for bounded $S \in \borelg$, the point process with support $[\Pi] \cap S$ is a
finite subprocess of $\Pi$.

Assume that \eqref{conditionS} fails, so that there exists  a bounded $S \in \borelg$  such that $\P( \Pi(S) = 0 \ | \ \Pi\ronnn_{S^c}) = 0$ on some set of positive measure.   Thus there exists $A \in \goth{M}$ such that
$\P(   \Pi(S) = 0, \ \Pi\ronnn_{S^c} \in A) = 0$ and $\P(\Pi\ronnn_{S^c} \in A) >0$.   With $A':=  A \cap \ns{\mu  \in \mathbb{M}: \mu(S) = 0}$, we have $\P( \Pi \ronnn_{S^c} \in A') >0$, but $\P(\Pi \in A') = 0$, so that $\Pi$ does not satisfy \eqref{S}.       
\end{proof}

We remark that the condition in Theorem 1 (\ref{S}, \ref{conditionS}) that $S$ is bounded may be replaced with the condition that $\Pi(S) < \infty$ a.s., and the resulting statements are also equivalent.   It is easy to verify that the modified statements $\eqref{S}',   \eqref{conditionS}'$ satisfy:  $\eqref{minusF} \Rightarrow \eqref{S}' \Rightarrow \eqref{conditionS}' \Rightarrow \eqref{conditionS}$.  In particular if $\Pi$ is translation-invariant and of finite intensity, then any $S$ of finite volume satisfies $\Pi(S) < \infty$ a.s.

For a translation $\theta$ of $\R^d$ and a point measure $\mu \in \XX$, we
define $\theta\mu\in\XX$ by $(\theta \mu)(S):= \mu(\theta^{-1} S)$ for all $S
\in \borelg$; for $\A \in \goth{M}$, we write $\theta \A := \ns{\theta \mu :
\mu \in \A}$. For $x \in \R^d$ let $\theta_x$ be the translation defined by
$\theta_x(y):=y+x$ for all $y \in \R^d$.

\begin{proof}[Proof of Proposition \ref{anyS}]
Let $U,V$ be uniformly distributed on  $S,T \in \borelg$ respectively and let
$U,V,\Pi$ be independent. Assume that $\Pi +\delta_U \prec \Pi$ and let $\A
\in \goth{M}$ be such that $\P(\Pi+ \delta_V \in \A) > 0$. We will show that
$\P(\Pi \in \A) > 0$.

Since $\Pi$ is translation-invariant, for all $\A' \in \goth{M}$ we have
$\P(\Pi +\delta_{\theta U} \in \A')= \P(\Pi + \delta_U \in \theta^{-1}\A')$
and thus $\Pi +\delta_{\theta U} \prec \Pi$ for all translations $\theta$ of
$\R^d$.   By Remark \ref{fubini}, $T':=\ns{w \in T: \P(\Pi + \delta_w \in \A)
> 0}$ has positive Lebesgue measure.  By the  Lebesgue density theorem
\cite[Corollary~2.14]{MR1333890}, there exist $x \in T'$, $y \in S$, and $\e
>0$ such that
\begin{align*}
\les(T \cap B(x,\e)) &> \tfrac12 \les B(x, \e); \\
\les(S \cap B(y, \e)) &> \tfrac12 \les B(y, \e).
\end{align*}

Thus with $z=x-y$, the set $T' \cap \theta_{z} S$ has positive Lebesgue
measure.  Thus by Remark \ref{fubini}, $\P(\Pi +\delta_{\theta_z U} \in \A) >
0$.   Since $\Pi +\delta_{\theta_z U} \prec \Pi$, we have $\P(\Pi \in \A)>0$.
\end{proof}

\section{Palm equivalences}
\label{palm}

In this section, we  discuss  insertion-tolerance and
deletion-tolerance in the context of Palm processes.  We begin by presenting
some standard definitions and facts.  Let $\Pi$ be a translation-invariant
point process with finite intensity $\lambda$. The Palm version of $\Pi$ is
the point process $\Pi^*$ such that for all $\A \in \goth{M}$ and all $S \in
\borelg$ with finite Lebesgue measure, we have
\begin{equation}
\label{palmeq}
\E\#\bigl\{x\in[\Pi]\cap S
\text{ with } \Pi\in \theta_x \A\bigr\}=\lambda\les S \cdot \P(\Pi^*\in \A),
\end{equation}
where $\#B$ denotes the cardinality of a set $B$. Sometimes (\ref{palmeq}) is
called the {\em{Palm property}}.

By a monotone class argument, a consequence of (\ref{palmeq}) is that for all
measurable $f:\XX \times \R^d \to [0,\infty)$ we have
\begin{equation}
\label{palmeqg}
\E \int_{\R^d} f(\theta_{-x}\Pi, x)  \,d\Pi(x) = \lambda
\int_{\R^d}\E f(  \Pi^{*}, x)  \,dx;
\end{equation}
see \cite[Chapter 11]{MR1876169}.

\begin{proof}[Proof of Theorem \ref{suff}]
Let $\Pi$ have intensity $\lambda > 0$.  Let $S\in \borelg$ have finite
Lebesgue measure.  By Theorem \ref{equiv} it suffices to show that
$\Pi\ronnn_{S^c} \prec \Pi$.

Let $\P(\Pi\ronnn_{S^c} \in \A) > 0$, for some $\A \in \goth{M}$.  Thus we
may assume that
\begin{equation}
\label{gzero}
\P(\exists x \in [\Pi] \cap S : \Pi -\delta_x \in \A) >0,
\end{equation}
otherwise $\P(\Pi \in \A) > 0$.  By applying (\ref{palmeqg}) to the function
$$(\mu, x) \mapsto \onea{ \mu -\delta_0 \in \theta_{-x}\A } \onea{x \in S},$$
we obtain
\begin{equation}
\label{palmcons}
\E\#\ns{x \in [\Pi ]\cap S : \theta_{-x}(\Pi - \delta_x) \in \theta_{-x} \A} \\
=\lambda \int_S \P(\Pi^{*} - \delta_0 \in \theta_{-x} \A) dx.
\end{equation}
From (\ref{gzero}) and (\ref{palmcons}), we deduce that $\P(\Pi^{*} -
\delta_0 \in \theta_{-x} \A)>0$, for some $x \in S$.  By assumption, $\P(\Pi
\in \theta_{-x} \A) > 0$.  Since $\Pi$ is translation-invariant, $\P(\Pi \in
\A) >0$.
\end{proof}

\begin{proof}[Proof of Theorem \ref{thm-instol-stat-eq}, (i) $\Rightarrow$ (\ref{original})]
Suppose that $\Pi + \delta_0$ is not absolutely continuous with
respect to $\Pi^*$; then there exists $\A \in \goth{M}$ such
that
$$\P(\Pi^*\in \A)=0 \quad\text{but}\quad \P(\Pi+\delta_0\in \A)>0.$$
Without loss of generality, take $\A$ to be a set that does not
care whether there is a point at $0$; that is if $\mu \in \A$,
then $\mu' \in \A$, provided $\mu,\mu'$ agree on
$\R^d\setminus\{0\}$.  By translation-invariance,
$$0<c:=\P(\Pi+\delta_0\in \A)=\P(\Pi\in \A)=\P(\Pi\in \theta_x \A)$$
for every $x\in\R^d$.  Hence the translation-invariant random set $G:=\{x\in
\R^d: \Pi\in \theta_x \A\}$ has intensity $\E \les([0,1]^d \cap G) =c$.
Moreover, if $U$ is uniformly distributed in $[0,1]^d$ and independent of
$\Pi$, then $\P(U \in G) =c$.   Therefore defining the set
$$\A':=\{\mu\in \XX: \exists x\in[\mu]\cap [0,1]^d
\text{ with } \mu\in \theta_x \A\},$$ we deduce that $\P(\Pi+\delta_U\in
\A')>0$. (Recall that $\A$ does not care whether there is a point at $0$.)
On the other hand by the Palm property (\ref{palmeq}) we have
\begin{eqnarray*}
\P(\Pi\in \A') &\leq& \E\#\{x\in[\Pi]\cap [0,1]^d
\text{ with } \Pi\in \theta_x \A\} \\ &=& \lambda\les S \cdot \P(\Pi^*\in \A)=0.
\end{eqnarray*}
Thus $\Pi$ is not insertion-tolerant.
\end{proof}

The following observations will be useful in the proof that (\ref{original})
implies (i) in  Theorem \ref{thm-instol-stat-eq}.
\begin{lemma}
\label{originalfubini} Let $\Pi$ be a translation-invariant point process on
$\R^d$ with finite intensity.  If $Y$ is any $\R^d$-valued random variable,
and $U$ is uniformly distributed in $S \in \borelg$ and independent of $(\Pi,
Y)$, then $\theta_U \theta_Y \Pi \prec \Pi$.
\end{lemma}

\begin{lemma}
\label{wextrahead} Let $\Pi$ be a translation-invariant point process on
$\R^d$ with finite intensity.  There exists a $\Pi$-point $Z$ such that
$\Pi^{*} \prec \theta_{-Z}\Pi$.
\end{lemma}

\begin{proof}[Proof of Theorem \ref{thm-instol-stat-eq}, (\ref{original}) $\Rightarrow$ (i)]
Suppose that $\Pi + \delta_0 \prec \Pi ^{*}$. Without loss of generality we
may assume that $\Pi$ and $\Pi^{*}$ are defined a common probability space.
By Lemma \ref{wextrahead}, there exists a $\Pi$-point $Z$ such that
\begin{equation}
\label{clearone}
\Pi^{*} \prec \theta_{-Z} \Pi.
\end{equation}

Let $U$ be uniformly distributed in a Borel set $S$ and   independent of
$(\Pi, \Pi ^{*}, Z)$.  By Lemma \ref{originalfubini}, it suffices to show
that $\Pi + \delta_U \prec \theta_U \theta_{-Z} \Pi$.    Since $U$ is
independent of $(\Pi, \Pi^{*},Z)$, from \eqref{clearone} it follows that
$\theta_U \Pi^{*} \prec \theta_U \theta_{-Z} \Pi$.   Thus it remains to show
that $\Pi + \delta_U \prec \theta_U \Pi^{*}$.

Since $\Pi$ is translation-invariant and $U$ is independent of $\Pi$ we have
\begin{equation}
\label{insU}
\theta_U (\Pi + \delta_0) \eqd \ \Pi + \delta_U.
\end{equation}
Since we assume that $\Pi + \delta_0 \prec \Pi ^{*}$ and $U$ is independent
of $(\Pi, \Pi^{*})$ we deduce from \eqref{insU} that $\Pi + \delta_U \prec
\theta_U \Pi ^{*}$.
\end{proof}

\begin{proof}[Proof of Lemma \ref{originalfubini}]
Let $Q$ be the joint law of $\Pi$ and $Y$.    Since $U$ is independent of
$(\Pi, Y)$, by Fubini's theorem, for all
$\A \in {\goth{M}}$, we have
\begin{align*}
\P(\theta_U \theta_{Y} \Pi \in \A)
&= \frac{1}{\les(S)}\int\left(\int_S \onea{\theta_{u+ y} \pi \in \A}
du \right) dQ(\pi, y) \\
& \leq  \frac{1}{\les(S)}\int \left(\int_{\R^d} \onea{\theta_x  \pi \in \A}
dx \right) dQ(\pi, y) \\
&= \frac{1}{\les(S)}\int_{\R^d} \P(\theta_x \Pi \in \A) dx\\
&= \frac{1}{\les(S)}\int_{\R^d} \P(\Pi \in \A) dx.
\qedhere
\end{align*}
\end{proof}

Lemma \ref{wextrahead} is an immediate consequence of a   result of Thorisson
\cite{Thorissontrv}, which states that there exists a {\em shift-coupling} of
$\Pi$ and  $\Pi^{*}$; that is, a $\Pi$-point $Z$ such that $\Pi^{*} \eqd
\theta_{-Z} \Pi$.  In fact, Holroyd and Peres \cite{Extra-Heads} prove that
such a $Z$ may be chosen as a deterministic function of $\Pi$.  Since
Lemma~\ref{wextrahead} is much weaker result, we can give the following
simple self-contained proof.
\begin{proof}[Proof of Lemma \ref{wextrahead}]
Let $\{a_i\}_{i \in \N} = [\Pi]$ be an enumeration of the $\Pi$-points. Let
$K$ be a random variable with support $\N$; also assume that $K$ is
independent of $(a_i)_{i \in \N}$.  Define the $\Pi$-point $Z := a_K$. We
will show that $\Pi^{*} \prec \theta_{-Z}\Pi$.

Let $\A \in \M$ be so that $\P( \Pi^{*} \in \A) > 0$.  By the Palm property
\eqref{palmeq}, there exists a $\Pi$-point $Z'= Z'(\A)$ such that $\P(
\theta_{-Z'} \Pi \in \A) > 0$; moreover,  there exists $i \in \N$ such that
$\P( \theta_{-Z'} \Pi \in \A,  \ Z'=a_i) > 0$.  Since $K$ is independent of
$(a_i)_{i \in \N}$, it follows from the definition of $Z$ that   $$\P(
\theta_{-Z'} \Pi \in \A, \; Z'=a_i,\; K=i,\; Z=a_i) > 0.$$  Therefore, $\P(
\theta_{-Z} \Pi \in \A) > 0$.
\end{proof}

\section{Continuum percolation}
\label{contperc}

Theorem \ref{percuniq} is an immediate consequence of the following.
Consider the Boolean continuum percolation model for a point process $\Pi$.
Let $W$ denote the cluster of containing the origin. For $M > 0$, an
{\df{$\boldsymbol{M}$-branch}} is an unbounded component of $W \cap B(0,
M)^c$.

\begin{lemma}
\label{choices} \sloppypar For a translation-invariant ergodic insertion-tolerant
point process, the number of unbounded clusters is a fixed
constant a.s.\ that is zero, one, or infinity.
\end{lemma}

\begin{lemma}
\label{three}  If an insertion-tolerant point process has  infinitely many
unbounded clusters, then with positive probability there exists $M >0$ so
that there at least three $M$-branches.
\end{lemma}
\begin{theorem}
\label{comb}  For all $M >0$,  a translation-invariant ergodic point process
has at most two $M$-branches.
\end{theorem}
For a proof of Theorem \ref{comb} see \cite[Theorem 7.1]{roy}.

\begin{proof}[Proof of Theorem \ref{percuniq}]
From Lemma \ref{choices}, it suffices to show that there can not be
infinitely many unbounded clusters; this follows from Theorem~\ref{comb} and
Lemma~\ref{three}.
\end{proof}

For $r > 0$, let  $r\Z^d:= \ns{rz : z\in \Z^d}$.
\begin{proof}[Proof of Lemma \ref{choices}]
Let $\Pi$ be a translation-invariant ergodic insertion-tolerant point
process.  Let the occupied region be given by a union of balls of radius $R
>0$.  By ergodicity, if $K(\Pi)$ is the number of unbounded clusters, then
$K(\Pi)$ is a fixed constant a.s.   Assume that $K(\Pi) < \infty$.  It
suffices to show that $\P(K(\Pi) \leq 1) >0$.  Since $K (\Pi)< \infty$, there
exists  $N >0$ so that  every unbounded cluster intersects $B(0,N)$ with
positive probability.  Consider the finite set $S:=(\fracc{R}{4})\Z^d \cap
B(0,N)$.  For each $x \in S$,    let $U_x$ be uniformly distributed in $B(x,
R)$ and assume that the $U_x$ and $\Pi$ are independent.     Let $\gothh{F}:=
\sum_{x \in S} \delta_{U_x}$.  Since $B(0,N) \subset \cup_ {x \in S} B(U_x,
R)$, we have  that $\P(K(\Pi + \gothh{F}) \leq 1) >0$.  By Theorem
\ref{thm-instol-eq} \eqref{finiteadd}, $\Pi + \gothh{F} \prec \Pi$, so that
$\P( K(\Pi) \leq 1) >0$.
\end{proof}

\begin{proof}[Proof of Lemma \ref{three}]
The proof is similar to that of Lemma {\ref{choices}}.  Let $\Pi$ be an
insertion-tolerant point process with infinitely many  unbounded clusters.
Let the occupied region be given by a union of balls of radius $R >0$. Choose
$N$ large enough so that  at least three unbounded clusters  interest
$B(0,N)$ with positive probability. Define a finite point process $\gothh{F}$
exactly as in the proof of Lemma \ref{choices}.  The point process $\Pi +
\gothh{F}$ has at least three $(N+R)$-branches with positive probability and
Theorem \ref{thm-instol-eq} \eqref{finiteadd} implies that  $\Pi + \gothh{F}
\prec \Pi$.  Thus $\Pi$ has at least three $(N+R)$-branches with positive
probability.
\end{proof}

\section{Stable matching}
\label{stableM}

Theorems \ref{onethm} and \ref{twothm} are  consequences of the following
lemmas.   Let $\gothh{R}$ be a point process with a unique one-colour stable
matching scheme $\gothh{M}$.  Define

\begin{equation}
\label{H}
H = H(\gothh{R}):=\bigl\{x \in [\gothh{R}] : \|x - \gothh{M}(x)\| > \|x\| -1\bigr\}.
\end{equation}
This is the set of $\gothh{R}$-points that would prefer some
$\gothh{R}$-point in the ball $B(0,1)$, if one were present in the
appropriate location, over their current partners.  Also define $H$ by
\eqref{H} for the case of two-colour stable matching.

A calculation given in \cite[Proof of Theorem 5(i)]{random} shows that, for
one-colour and two-colour matchings,
\begin{equation}
\label{Hthmfive}
\E \#H  = c \, \E^{*} \big[(\gothhh{X} +1)^d\big].
\end{equation}
for some $c=c(d)\in(0,\infty)$.

\begin{lemma}[One-colour stable matching]
\label{oneH}  Let $\gothh{R}$ be a translation-in\-variant point process on
$\R^d$ with finite intensity that almost surely  is non-equidistant and has
no descending chains. If $\gothh{R}$ is insertion-tolerant, then $\P(\#H =
\infty) =1$.  If $\gothh{R}$ is deletion-tolerant, then \linebreak $\P(\#H =
\infty)>0$.
\end{lemma}

\begin{lemma}[Two-colour stable matching]
\label{twoH} Let $\gothh{R}$ and $\gothh{B}$ be independent
translation-invariant ergodic point processes on $\R^d$ with equal finite
intensity, such that the point process $\gothh{R} +\gothh{B}$ is
non-equidistant and has no descending chains.  If $\gothh{R}$  is
insertion-tolerant, then $\P(\# H = \infty) =1$.    If $\gothh{R}$  is
deletion-tolerant, then $\P(\#H = \infty) >0$.
\end{lemma}
\begin{remark}
\label{prime} Recall that in the case of two-colour stable matching we
defined $X$ in terms of the distance from an $\gothh{R}$-point to its
partner. If we instead define $X'$ by replacing $\gothh{R}$ with $\gothh{B}$
in \eqref{defdist}, then $X'\eqd X$; see the discussion after
\cite[Proposition 7]{random} for details. \egg
\end{remark}

\begin{proof}[Proof of Theorem \ref{onethm}]
 Use Lemma \ref{oneH}  together with \eqref{Hthmfive}.
\end{proof}

\begin{proof}[Proof of Theorem \ref{twothm}]
Use Lemma \ref{twoH}  together with \eqref{Hthmfive} and Remark~\ref{prime}.
\end{proof}

The following lemmas concerning stable matchings in a deterministic setting
will be needed.   A {\df partial matching} of a point measure $\mu \in \XX$
is the edge set $m$ of simple graph $([\mu], m)$ in which every vertex has
degree at most one.  A partial matching is a {\df perfect} matching if every
vertex has degree exactly one.   We write $m(x) = y$ if and only if $\ns{x,y}
\in m$, and set $m(x) = \infty$ if $x$ is unmatched.   We say a partial
matching is {\df{stable}} if  there do not exist distinct points $x,y \in
[\mu]$ satisfying
\begin{equation}
\label{defstableb}
\|x-y\| < \min\ns{ \|x - m(x)\|, \|y - m(y)\|},
\end{equation}
where $\| x - m(x)\| = \infty$ if $x$ is unmatched.   Note that in any stable
partial matching there can be at most one unmatched point.

For each $\e >0$, set
$$H_{\e} = H_{\e}(\mu) :=\ns{x \in [\mu] : \|x -  m(x)\| > \|x\| -\e}.$$
For each $y\in \R^d$, set
 $$N(\mu, y) := \ns{x \in [\mu] \setminus \ns{y}: \|x - m(x)\| > \|x - y\|}.$$
This is the set of $\mu$-points that would prefer $y \in \R^d$ over their partners.

\begin{lemma}
\label{monohk} If $\mu\in \XX$ is non-equidistant and has no descending
chains, then $\mu$ has an unique stable partial matching $m$. In addition, we
have the following properties.

\begin{enumerate}[(i)]
\item \label{del} If $\ns{x,y} \in m$ is a matched pair, then $m
    \setminus\ns{\ns{x,y}}$ is the unique stable partial matching of $\mu
    -\delta_x - \delta_y$.
\item \label{ins} Let $\e > 0$.  If $m$ is a perfect matching and
    $\#H_{\e}=0$, then for all $x \in B(0, \e)$ such that $\mu +
    \delta_x$ is non-equidistant,   $m$ is the unique stable  partial
    matching of $\mu + \delta_x$; in particular, $x$ is unmatched in $m$.
\item \label{Ndel} If $\ns{x, y} \in m$ is a matched pair and $\# N(\mu,
    y) =0$, then $m\setminus \nss{x,y}$ is the unique stable partial
    matching of $\mu - \delta_x$, and in particular, $y$ is left
unmatched.
\end{enumerate}
\end{lemma}
\begin{proof}
The existence and uniqueness is given by \cite[Lemma 15]{random}. Thus for
(i)--(iii) it suffices to check that the claimed matching is stable, which is
immediate from the definition \eqref{defstableb}.
\end{proof}

The next lemma is a simple consequence of Lemma \ref{monohk}.
\begin{lemma}
\label{addremove} Suppose that $\mu \in \XX$ is non-equidistant and has no
descending chains.  Let $m$ be the unique stable matching of $\mu$.  Suppose
that $\ns{x,y} \in m$ and $0 \not \in [\mu]$.  There exists  $\e >0$ such
that for $\les$-a.a.\ $x' \in B(x, \e)$ and  $y' \in B(y, \e)$:   the unique
stable matching  $m'$ of $\mu + \delta_{x'} + \delta_{y'}$   is given by $$
m' =(m \setminus \nss{x,y}) \cup \ns{\ns{x,x'} , \ns{y,y'}},$$ and
furthermore, ${x,x'}, {y,y'} \not \in H_{\e}(\mu + \delta_{x'} + \delta_{y'})
\subseteq H_{\e}(\mu)$.
\end{lemma}

\begin{proof}[Proof of Lemma \ref{addremove}]
Consider $$d_v:= \min\ns{ \|v-w\|:  w \in [\mu] \cup \ns{0}, \  w \not = v }.$$   Let
\begin{equation}
\label{defep}
\e:= \tfrac15 \min \ns{ d_x, d_y, d_0}
\end{equation}
(any multiplicative factor less than $\tfrac14$ would suffice here).  Let
$A:= B(x, \e) \times B(y, \e)$. It is easy to verify that for $\les$-a.a.\
$(x',y') \in A$ that the measure $\mu + \delta_{x'} + \delta_{y'}$ is also
non-equidistant and has no descending chains.  Thus by Lemma \ref{monohk},
for $\les$-a.a.\ $(x',y') \in A$ the measure $\mu + \delta_{x'} +
\delta_{y'}$ has a unique stable perfect matching $m'$. Clearly, by
\eqref{defep} and \eqref{defstableb}, we have that $\ns{x,x'}, \ns{y,y'} \in
m'$.  On the other hand,  by Lemma~\ref{monohk} \eqref{del}, $m' \setminus
\ns{\ns{x,x'},\ns{y,y'}}$ is the unique stable perfect matching of $\mu -
\delta_x - \delta_y$ and $m \setminus \nss{x,y}$ is the also the unique
stable perfect matching of $\mu - \delta_x - \delta_y$.  Thus
$$m'= (m \setminus \nss{x,y}) \cup \ns{\ns{x,x'} , \ns{y,y'}}.$$
It also follows from  \eqref{defep}  that
\begin{equation*}
 {x,x'},{y,y'} \not \in  H_{\e}(\mu + \delta_{x'} + \delta_{y'}) \subseteq H_{\e}(\mu).
\qedhere
\end{equation*}
\end{proof}

\begin{proof}[Proof of Lemma \ref{oneH}: the case where $\gothh{R}$ is insertion-tolerant]
Let $\gothh{R}$ be insertion-tolerant. Note that $H_1(\gothh{R}) =
H(\gothh{R})$. First, we will show that
 \begin{equation}
 \label{hkthing}
 \P( \#H_{\e}(\gothh{R}) > 0) =1 \ \text{for all} \  \e >0.
 \end{equation}

Second, we will show that if $\P( 0<\# H_1(\gothh{R})<\infty)
>0$, then there exists a finite point process $\F$ such that $\F$ admits a
nice conditional law given $\gothh{R}$, and
\begin{equation}
\label{limithk}
\lim_{ \e \to 0} \P\bigl( \#H_{\e}(\gothh{R} + \F) = 0\bigr)
= \P\bigl( 0<\# H_1(\gothh{R})<\infty\bigr) > 0.
\end{equation}

Finally, note that by Corollary \ref{weak} and the insertion-tolerance of
$\gothh{R}$ that   \eqref{limithk} and \eqref{hkthing} are in contradiction.
Thus $\P( \#H_1(\gothh{R}) = \infty) = 1$.  It remains to prove the first two
assertions.

The following definition will be useful.  Let $\XX'$ be the set of point
measures $\mu \in \XX$ such that $\mu$  has  a unique stable perfect
matching, has no descending chains, and is   non-equidistant.

Let $\e > 0$. Let $\J$ be the set of point measures $\mu \in \XX'$ such that
$\#H_{\e}(\mu) = 0$.    To show \eqref{hkthing}, it suffices to prove that
$\P(\gothh{R} \in \J) = 0$.    Let $\mu \in \J$ and let $m$ be the unique
stable perfect matching for $\mu$.   By Lemma \ref{monohk} (\ref{ins}),  for
Lebesgue-a.a.\ $x \in B(0, \e)$ the unique stable partial matching  for $\mu
+ \delta_{x}$ is $m$ (and $x$ is unmatched). If $\P(\gothh{R} \in \J) > 0$,
then it follows from the insertion-tolerance of $\gothh{R}$ that with
positive probability $\gothh{R}$ does not have a perfect stable matching, a
contradiction.

Now let $\A$ be the set of point measures $\mu \in \XX'$ such that $0<
\#H_1(\mu) < \infty$ and $0 \not \in [\mu]$.
 If $\gothh{R} \in \A$, then, by applying Lemma
\ref{addremove} repeatedly, there exists $\rho = \rho(\gothh{R})$ such that
if a point is added within distance $\rho$ of each point in $H_1(\gothh{R})$
and each of their partners, then (for $\leb$-a.a.\ choices of such points)
the resulting process $\gothh{R}'$ satisfies $H_\rho(\gothh{R}') = 0$.  Let
$\F$ be the finite point process whose conditional law given $\gothh{R}$ is
given as follows.  Take independent uniformly random points in each of the
appropriate balls of radius $\rho$ provided $\gothh{R}\in\A$; otherwise take
$\F= 0$.  By the construction,
\begin{equation*}
\lim_{\e \to 0} \P\big( \#H_\e(\gothh{R} + \F)=0  \mid
 \gothh{R} \in \A, \;   \rho(\gothh{R}) > \e \big)=1,
\end{equation*}
so \eqref{limithk} follows.
\end{proof}

\begin{proof}[Proof of Lemma \ref{oneH}: the case where $\gothh{R}$ is deletion-tolerant]
Suppose $\gothh{R}$ is deletion-tolerant.    We will show that for any
$\gothh{R}$-point $Z$
\begin{equation}
\label{N}
\#N(\gothh{R},Z) = \infty \   \text{a.s.}
 \end{equation}
From (\ref{N}) it follows that if $\gothh{R}(B(0,1))> 0$, then $\#H =
\infty$.  Since $\gothh{R}$ is translation-invariant, $\P(
\gothh{R}(B(0,1)) > 0) > 0$ and $\P(\#H = \infty) > 0$.

It remains to show (\ref{N}).  Let $Z$ be an $\gothh{R}$-point. Let $\F_1$ be
the point process with support $N(\gothh{R},Z)$, and let $\F_2$ be the point
process with support $\ns{\gothh{M}(y):  y \in N(\gothh{R},Z)}$.  Consider
the point process $\F$ defined by
\begin{eqnarray*}
\F &:=& \begin{cases}
 \F_1 + \F_2, \ \
\text{if} \ \ \#{N}(\gothh{R},Z) < \infty \\
0, \ \  \text{otherwise}.
\end{cases}
\end{eqnarray*}
Let $\gothh{M}'$ be  given by $$[\gothh{M}']:=[\gothh{M}]  \setminus
\bigcup_{x \in [\F]} \nss{x, \gothh{M}(x)}.$$ By Lemma \ref{monohk}
\eqref{del},  $\gothh{M'}$   is the unique stable matching  for $\gothh{R} -
\F$ a.s.

Towards a contradiction assume that $\P( \#{N}(\gothh{R},Z) < \infty) > 0$.
Thus,     $\P (N(\gothh{R} - \F, Z) = 0) > 0$.  By Lemma \ref{monohk}
\eqref{Ndel},   with  positive probability, $\gothh{R} - \F
-\delta_{\gothh{M}(Z)}$ has the unique stable  partial matching  given by
$\gothh{M}'$ with the pair $\ns{Z, \gothh{M}(Z)}$ removed and $Z$ left
unmatched.   From Theorem~\ref{equiv}~\eqref{minusF} and the
deletion-tolerance of $\gothh{R}$ we have $\gothh{R} - \F
-\delta_{\gothh{M}(Z)} \prec \gothh{R}$.   Thus with positive probability,
$\gothh{R}$ has a stable partial matching with an unmatched point, a
contradiction.
\end{proof}

We now turn to the two-colour case.  Given two point measures $\mu, \mu' \in
\XX$ such that $\mu + \mu'$ is a simple point measure,  we say that $m$ is a
{\df{partial}} (respectively, {\df{perfect)}} matching of $(\mu, \mu')$ if
$m$ is the edge set of a simple bipartite graph $([\mu], [\mu'], m)$ in which
every vertex has degree at most one (respectively, exactly one). We write
$m(x) = y$ if and only if $\ns{x,y} \in m$ and set $m(x)= \infty$ if $x$ is
unmatched.  We say that $m$ is {\df{stable}} if there do not exist $x \in
[\mu]$ and $y \in [\mu']$ satisfying \eqref{defstableb}.  If $\mu + \mu'$ is
non-equidistant and has no descending chains then there exists a unique
stable partial matching of $(\mu, \mu')$ \cite[Lemma 15]{random}.

\begin{remark}
\label{transfer} \sloppypar It is easy to verify that the two-colour
analogues of Lemma~\ref{monohk} \eqref{del} and \eqref{Ndel} hold. \egg
\end{remark}

We will need the following monotonicity facts about stable two-colour
matchings.  Similar results are proved in \cite[Proposition 21]{Stable-PL},
\cite{galeshapley}, and \cite{MR1415126}.

\begin{lemma}
\label{monotrick} Let $\mu, \mu \in \XX$ and assume that $\mu + \mu'$ is a
simple point measure that is non-equidistant and has no descending chains.
Let  $m$ be the   stable partial matching of $(\mu, \mu')$.
\begin{enumerate}[(i)]
\item \label{previous} Assume that  $w \not \in [\mu']$ and $\mu' +
    \delta_w$ is  non-equidistant and has no descending chains. If $m'$
    is the stable partial matching of $(\mu, \mu' + \delta_w)$, then
$$\|z - m(z) \| \geq \|z - m'(z)\| \ \text{for all} \  z \in [\mu].$$
\item\label{new} Let $x \in [\mu]$.   If $m'$ is the stable partial
    matching of $(\mu - \delta_x, \mu')$, then
\begin{equation}
\label{monotwocol}
\|z - m(z) \| \geq \|z - m'(z)\| \ \text{for all} \  z \in [\mu - \delta_x].
\end{equation}
\end{enumerate}
\end{lemma}

\begin{proof}[Proof of Lemma \ref{monotrick}]
Part \eqref{previous} follows from  \cite[Lemma 17]{random}.  For part
\eqref{new}, if $x$ is not matched under $m$, then $m' = m$,  thus assume
that $m(x) = y$.    By Lemma \ref{monohk} \eqref{del} and Remark
\ref{transfer}, $m \setminus \nss{x,y}$ is the unique stable partial matching
for $(\mu - \delta_x, \mu' - \delta_y)$.  Thus by part \eqref{previous},
$m'$, the unique stable matching for $(\mu -\delta_x, \mu')$, satisfies
\eqref{monotwocol}.
\end{proof}

\begin{proof}[Proof of Lemma \ref{twoH}]
The proof for the case when $\gothh{R}$ is ins\-ert\-ion-tolerant is given in
\cite[Theorem 6(i)]{random}.  In the case when $\gothh{R}$ is
deletion-tolerant we proceed similarly to the proof of Lemma \ref{oneH}.
Recall that in the two-colour case, $\M$ denotes the two-colour stable
matching scheme for $\gothh{R}$ and $\gothh{B}$.   Let $Z$ be a
$\gothh{B}$-point. Define $N(\gothh{R}, Z)$ and $\F_1$ as in the proof of
Lemma \ref{oneH}, so that $N(\gothh{R}, Z)$ is the set of $\gothh{R}$-points
that would prefer $Z$ over their partners and $\F_1$ is the point process
with support $N(\gothh{R}, Z)$.

Towards a contradiction assume that $\P( \#N(\gothh{R}, Z) < \infty) >0$.
There exists a  unique stable partial matching for  $(\gothh{R}- \F_1,
\gothh{B})$ a.s.; denote it by $\M'$.  From Lemma \ref{monotrick}
\eqref{new}, it follows that
\begin{equation}
\label{newNR}
\P( N(\gothh{R} - \F_1, Z) = 0) >0.
\end{equation}
From \eqref{newNR} and Remark \ref{transfer} with Lemma \ref{monohk}
\eqref{Ndel}, it follows that with positive probability,   $\M' \setminus
\ns{\ns{Z, \M'(Z)}}$ is the unique stable partial matching for $(\gothh{R} -
\F_1 - \M'(Z), \gothh{B})$ and the $\gothh{B}$-point $Z$ is left unmatched.
By Lemma~\ref{unipick}, there exists a bounded Borel set $S$ with finite Lebesgue
measure such that $\P(\gothh{R} \ronnn_S = \F_1 + \delta_{\M'(Z)}) > 0$.  By
Theorem \ref{equiv} \eqref{S} and the deletion-tolerance of $\gothh{R}$, we
have that $\gothh{R} \ronnn_{S^c} \prec \gothh{R}$; furthermore, since
$\gothh{R}$ and $\gothh{B}$ are independent, $(\gothh{R} \ronnn_{S^c},
\gothh{B}) \prec (\gothh{R}, \gothh{B})$.  Thus with positive probability
$(\gothh{R}, \gothh{B})$ has a stable partial matching with a unmatched
$\gothh{B}$-point. This contradicts the fact that $\M$ is the two-colour
matching scheme for $\gothh{R}$ and  $\gothh{B}$.
\end{proof}

\section{Perturbed lattices and Gaussian zeros}
\label{perproof}

\subsection{Low-fluctuation processes}

\sloppypar Propositions \ref{pertone} and \ref{GAFplane}  will be proved
using the following more general result, which states that processes
satisfying various ``low-fluctuation'' conditions are neither insertion-tolerant
nor deletion-tolerant.   For a point process $\Pi$ and a measurable function
$h:\R^d \to \R$ write
\begin{equation}
\Pi(h):= \int  h(x) d\Pi(x) =\sum_{x \in [\Pi]} h(x).
\end{equation}
Let $\overline{B}(0,1) := \ns{x\in \R^d : \| x\| \leq 1}$ denote the closed unit ball.
\begin{proposition}[Low-fluctuation processes]
\label{decayo} Let $\Pi$ be a point process on $\R^d$ with finite intensity.
Let $h: \R^d \to [0,1]$ be a measurable function with $h(x) = 1$ for all $x
\in B(0,1/2)$ and support in $\overline{B}(0,1)$.  For each $n \in \Z^{+}$, set $h_n(x)
:= h(x/n)$ for all $x \in \R^d$. \enlargethispage*{5mm}
\begin{enumerate}[(i)]
\item
\label{ch-decaya}
If $\Pi(h_n) - \E \Pi(h_n)   \to 0$ in probability as $n \to \infty$,
then $\Pi$ is neither insertion-tolerant nor deletion-tolerant.
\item
\label{orseq}
If there exists a deterministic sequence $(n_k)$ with $n_k \to \infty$ such that
\begin{equation}
\label{ceas}
\frac{1}{K}\sum_{k=1} ^K
\big( \Pi(h_{n_k}) - \E \Pi (h_{n_k})  \big) \prob  0 \quad\text{as }K\to\infty,
\end{equation} \enlargethispage*{5mm}
then $\Pi$ is neither insertion-tolerant nor deletion-tolerant.
\item \label{ch-decayb} Write $N_n = \Pi(h_n) - \E \Pi(h_n)$. If there
    exists a deterministic sequence $(n_k)$ with $n_k \to \infty$ and a
    discrete real-valued random variable $N$ such that for all $\ell \in
    \R$,
\begin{equation}
\label{ceastwo}
\frac{1}{K}\sum_{k=1} ^K \mathbf{1}[N_{n_k} \leq \ell] \
\stackrel{{\P}}{\to} \  \P(N \leq\ell)   \quad  \text{as }   K \to \infty,
\end{equation}
then $\Pi$ is neither insertion-tolerant nor deletion-tolerant.
\end{enumerate}
\end{proposition}

In our application of Proposition \ref{decayo} \eqref{ch-decayb}, $N_n$ will
be integer-valued (see \eqref{integervalued} below).

\begin{proof}[Proof of Proposition \ref{decayo} \eqref{ch-decaya}]
Let $m_n:= \E \Pi(h_n)$.  Since $\Pi(h_n) - m_n \to 0$ in probability, there
exists a (deterministic) subsequence ${n_k}$ such that $\Pi(h_{n_k}) -
m_{n_k} \to 0$ a.s.  On the other hand, if $U$ is uniformly distributed in
$B(0,1)$, then $(\Pi + \delta_U)(h_{n_k}) - m_{n_k} \to 1$ a.s.  Therefore
$\Pi$ is not insertion-tolerant.  Similarly, if $Z$ any $\Pi$-point,
then $(\Pi - \delta_Z)(h_{n_k})-m_{n_k} \to -1$.  So $\Pi$ is not
deletion-tolerant.
\end{proof}

\begin{proof}[Proof of Proposition \ref{decayo} \eqref{orseq}]
  Suppose that \eqref{ceas} holds for some deterministic sequence $(n_k)$. Let
$m_{n_k}:= \E \Pi(h_{n_k})$, and for each integer $K > 0$ define $S_K:\XX \to \R$ by
$$S_K(\mu):=\frac{1}{K}\sum_{k=1}^K   (\mu(h_{n_k}) - m_{n_k}).$$
Thus $S_K(\Pi) \to 0$ in probability as $K \to \infty$, and there exists a
subsequence $(K_i)$ so that $S_{K_i}(\Pi) \to 0$ a.s.
However, if  $U$ is uniformly distributed in  $B(0,1)$, then $S_{K_i}(\Pi +
\delta_U)\to 1$ a.s.  Thus $\Pi$ cannot be insertion-tolerant.  Similarly, if
$Z$ is a $\Pi$-point, then $S_{K_i}(\Pi  - \delta_Z) \to -1$ a.s. Thus $\Pi$
cannot be deletion-tolerant.
\end{proof}

\begin{proof}[Proof of Proposition \ref{decayo} \eqref{ch-decayb}]
Suppose that \eqref{ceastwo} holds for some deterministic sequence $(n_k)$
and some discrete random variable $N$.  Let $m_{n_k}:= \E \Pi(h_{n_k})$, and
let $N_{n_k}(\mu) := \mu(h_n) - m_{n_k}$ for all $\mu \in \XX$. For each
integer $K>0$, define $F_K: \XX \times \R \to [0,1]$ by
$$F_K(\mu, \ell) := \frac{1}{K}\sum_{k=1} ^K \mathbf{1}[N_{n_k}(\mu) \leq \ell].$$
Thus $F_K(\Pi, \ell) \to \P(N \leq \ell)$ in probability as $K \to \infty$
for all $\ell \in \R$.   Since $N$ is discrete and has countable support, by
a standard diagonal argument, there exists a subsequence $(K_i)$ so that
$F_{K_i}(\Pi, \ell) \to \P(N \leq \ell)$ a.s.\ for all $\ell \in \R$.   Fix
$a \in \R$ such that $\P(N \leq a) \not = \P(N \leq  a+1)$.    We have
$F_{K_i}(\Pi, a) \to \P(N \leq a)$ a.s.\ and $F_{K_i}(\Pi, a+1) \to \P(N \leq
a+1) $ a.s.  However,   if  $U$ is uniformly distributed in  $B(0,1)$, then
$F_{K_i}(\Pi + \delta_U, a+1)\to \P(N \leq a)$ a.s.  Thus $\Pi$ cannot be
insertion-tolerant.     Similarly, if $Z$ is a $\Pi$-point, then $F_{K_i}(\Pi
- \delta_Z, a) \to \P(N \leq  a+1)$ a.s. Thus $\Pi$ cannot be
deletion-tolerant.
\end{proof}

\subsection{Gaussian zeros in the plane}

\begin{proof}[Proof of Proposition \ref{GAFplane}]
Let $\GZF$ be the Gaussian zero process on the plane.  Sodin and Tsirelson
\cite[Equation (0.6)]{MR2121537} show that $\GZF$ satisfies the conditions of
Proposition \ref{decayo} \eqref{ch-decaya}, with a twice differentiable
function $h$; in particular they show that $\var \GZF(h_n) \to 0$ as $n \to
\infty$.   Hence $\GZF$ is neither insertion-tolerant nor deletion-tolerant.
\end{proof}

\subsection{Perturbed lattices in dimension $2$}

The proof of Proposition~\ref{pertone} for the case $d=2$ relies on the
following lemma.

\begin{lemma}
\label{masterlemma} Let $(Y_z:z\in\Z^2)$ be i.i.d.\ $\R^2$-valued random
variables with $\E Y_0=0$ and $\var \|Y_0\|=\sigma^2<\infty$.  Let $\Lambda$
be the point process given by $[\Lambda]:= \ns{z + Y_z: z \in \Z^2}$.  Let
$h:\R^2\to[0,1]$ have support in $B(0,1)$, and have Lipschitz constant at
most $c<\infty$, and let  $h(x) =1$ for all $x \in B(0,1/2)$.  Define
$h_r(x):=h(x/r)$ for $x\in\R^2$ and $r>0$.  Set $m_r:= \E  \Lambda(h_{r})$.
\begin{enumerate}[(i)]
\item
\label{finitevar}
For all $r>0$ we have $\var \Lambda(h_r) \leq C,$
for some $C=C(\sigma^2,c)<\infty.$

\item
\label{covdecay}
For all $r >0$, we have
$ \cov(\Lambda(h_r), \Lambda(h_R)) \to 0$ as $R \to \infty$.

\item \label{orseqtwo} There exists a deterministic sequence $(n_k)$ with
    $n_k \to \infty$ such that \eqref{ceas} is satisfied with $\Lambda$
    in place of $\Pi$; that is,
    $$\frac{1}{K}\sum_{k=1} ^K
\big( \Lambda(h_{n_k}) - \E \Lambda (h_{n_k})  \big) \prob  0 \quad\text{as }K\to\infty.$$
\end{enumerate}
\end{lemma}
Lemma \ref{masterlemma} parts \eqref{finitevar} and \eqref{covdecay} will
allow us to use a weak law of large numbers to  prove \eqref{orseqtwo}.

\begin{proof}[Proof of Proposition \ref{pertone} ($d=2$)]
We may clearly assume without loss of generality that $\E Y_0=0$.  Now apply
Lemma \ref{masterlemma} \eqref{orseqtwo} together with Proposition
\ref{decayo} \eqref{orseq}.
\end{proof}

\begin{proof}[Proof of Lemma \ref{masterlemma} \eqref{finitevar}]
Note that
\begin{equation}
\label{form}
\Lambda(h_r) =  \sum_{z\in\Z^2} h_r(z+Y_z).
\end{equation}
Thus by the independence of the $Y_z$,  we have
\begin{equation}
\label{sumone}
\var \Lambda(h_r) = \sum_{z\in\Z^2} \var h_r(z+Y_z);
\end{equation}
we will split this sum into two parts.  We write
$C_1,C_2$ for constants depending only on $\sigma^2$ and $c$.

Firstly, since $h_r$ has Lipschitz constant at most $c/r$, we
have for all $z\in \Z^2$,
$$\var h_r (z+Y_z)\leq \E [(h_r (z+Y_z)-h_r (z))^2]\leq \E[(c\|Y_z\|/r)^2]=(c\sigma/r)^2,$$
therefore
\begin{equation}
\label{sumtwo}
\sum_{z\in\Z^2:\\ \|z\|\leq 2r} \var h_r(z+Y_z)\leq C_1.
\end{equation}
Secondly, since $h_r$ has support in $B(0,r)$,
\begin{align*}
\var h_r(z+Y_z) &\leq \E [h_r(z+Y_z)^2] \\ &\leq \P[z+Y_z\in B(0,r)] = \P[Y_0\in B(-z,r)],
\end{align*}
therefore
\begin{align}
\label{sumthree}
\sum_{{z\in\Z^2: \|z\|> 2r}} \var h_r(z+Y_z) &\leq
\sum_{{z\in\Z^2: \|z\|> 2r}} \P[Y_0\in B(-z,r)] \nonumber \\ &\leq
C_2 r^2 \P(\|Y_0\|>r)\leq C_2 \sigma^2.
\end{align}
 The result now follows by combining \eqref{sumone}--\eqref{sumthree}.
\end{proof}

\begin{proof}[Proof of Lemma \ref{masterlemma} \eqref{covdecay}]
Note that by Lemma \ref{masterlemma} \eqref{finitevar}, for all $r, R > 0$,
we have that  $\cov(\Lambda(h_r), \Lambda(h_R)) < \infty.$  By \eqref{form}
and independence of the $Y_z$ we have
\begin{align*}
\cov(\Lambda(h_r), \Lambda(h_R)) =&\,
\E \Big(\sum_{z \in \Z^2}  h_r(z+Y_z) \; h_R(z + Y_z)\Big)
\\ &-\sum_{z \in \Z^2} \E h_r(z+Y_z) \;\E h_R(z + Y_z).
\end{align*}
Let $R > 2r$.   If $h_r(z+ Y_z) >0$, then $h_R(z + Y_z) =1$; thus
$$\cov(\Lambda(h_r), \Lambda(h_R)) =
m_r - \sum_{z \in \Z^2} \E h_r(z+Y_z) \;\E h_R(z + Y_z).$$
Since $ h_R \uparrow  1$ as $R \to \infty$, for each $z \in \Z^2$  we have by
the monotone convergence theorem that $\E h_R(z + Y_z) \uparrow 1$ as $R \to
\infty$.  An additional application of the monotone convergence theorem shows
that
\begin{equation*}
\lim_{R \to \infty} \sum_{z \in \Z^2} \E h_r(z+Y_z) \; \E h_R(z + Y_z)
=  \sum_{z \in \Z^2} \E h_r(z+Y_z) =m_r.
\qedhere
\end{equation*}
\end{proof}

We will employ the  following weak law of large numbers for dependent
sequences to prove Lemma \ref{masterlemma} \eqref{done}.

\begin{lemma}
\label{durrett} Let $Z_1, Z_2, \ldots$ be real-valued random variables with
finite second moments and zero means.  If   there exists a sequence $b(k)$
with $b(k) \to 0$ as $k \to \infty$ such that $\E (Z_n Z_m) \leq b(n-m)$ for
all $n \geq m$, then $(Z_1 + \cdots + Z_n)/n \to 0$ in probability as $n \to
\infty$.
\end{lemma}

Lemma \ref{durrett} is a straightforward generalization of the standard $L^2$
weak law.  See \cite[Chapter 1, Theorem 5.2 and Exercise 5.2]{MR1609153}.

\begin{corollary}
\label{durrettm} Let $Z_1, Z_2, \ldots$ be real-valued random variables with
finite second moments and zero means.   Suppose that there exists $C >0$,
such that $\E |Z_m|^2 \leq C$ for all $m \in \Z^{+}$.    If for all $m \in
\Z^{+}$ we have  $\E(Z_mZ_n) \to 0$ as $n \to \infty$, then there exists an
increasing sequence of positive of integers $(r_n)$ such that $(Z_{r_1} +
\cdots+ Z_{r_n})/n \to 0$ in probability as $n \to \infty$.   Furthermore,
for any further subsequence $(r_{n_k})$ we have $(Z_{r_{n_1}} + \cdots
+Z_{r_{n_k}})/k \to 0$ in probability as $n \to \infty$.
\end{corollary}

\begin{proof}
Consider the sequence $b(k):=1/k$, where we set $b(0) = C$. We will show that
there exists a sequence $r_k$ so that  $\E( Z_{r_n} Z_{r_m}) \leq 1/n$ for
all $n >m$.  Thus $Z_{r_k}$ satisfies the conditions of Lemma \ref{durrett}
with $b(k)$. We proceed by induction. Set $r_1 = 1$.  Suppose that $r_2,
\ldots, r_{k-1}$ have already been defined and satisfy $\E(Z_{r_n} Z_{r_m})
\leq 1/n$ for all $1 \leq m < n \leq k-1$.  It follows from Lemma
\ref{masterlemma} \eqref{covdecay} that there exists an integer $R > 0$ such
that $\E(Z_{r_m}Z_R) \leq 1/k$ for all $1 \leq m \leq k-1$; set $r_k:= R$.
Furthermore, if $(r_{n_k})$ is a subsequence of $(r_n)$, we have that if $m
<k$, then $\E(Z_{r_{n_m}} Z_{r_{n_k}}) \leq \fracc{1}{n_k} \leq
\fracc{1}{k}.$   Thus $Z_{r_{n_k}}$ satisfies the conditions of Lemma
\ref{durrett} with $b(k)$.
\end{proof}

\begin{proof}[Proof of Lemma \ref{masterlemma} \eqref{orseqtwo}]
For each $n \in \Z^{+}$, set $Z_{n} := \Lambda(h_{n}) - m_{n}$.  By  Lemma
\ref{masterlemma} parts \eqref{finitevar} and \eqref{covdecay}, $Z_n$
satisfies the conditions of Corollary~\ref{durrettm}.
\end{proof}

\subsection{Perturbed lattices in dimension $1$}

The proof of Proposition~\ref{pertone} for the case $d=1$ relies on the
following lemma.

\begin{lemma}
\label{mlemma} Let $(Y_z: z \in \Z)$ be i.i.d.\ $\R$-valued random variables.
Let $\Lambda$ be the point process given by $[\Lambda]:= \ns{z + Y_z: z\in
\Z}$.   Define $h(x) := \one{(-1, 1]}(x)$ for all $x \in \R$ and set $h_n(x)
:= h(x/n)$ for $x \in \R$ and $n \in \Z^{+}$. For each $n \in \Z^{+}$, let
$N_n := \Lambda(h_n) - \E \Lambda(h_n).$        Assume that $\E |Y_0| <
\infty$.
\begin{enumerate}[(i)]
\item
\label{tight} The family of random variables
$(N_n)_{n\in \Z^+}$ is tight and integer-valued.
\item \label{cov} For any $k,\ell \in \R$ and $a \in \Z^{+}$,
$$\P(N_a \leq k, N_n \leq \ell) - \P(N_a \leq k) \;
\P(N_n \leq \ell) \to 0 \quad\text{as} \ n\to\infty.$$
\item \label{done} There exists a deterministic sequence $(n_k)$ with
    $n_k \to \infty$ and an integer-valued random variable $N$ such that
    \eqref{ceastwo} is satisfied; that is, for all $\ell \in \R$,
$$\frac{1}{K}\sum_{k=1} ^K \mathbf{1}[N_{n_k} \leq \ell] \
\stackrel{{\P}}{\to} \  \P(N \leq \ell)   \quad  \text{as }   K \to \infty.$$
\end{enumerate}
\end{lemma}
As in the case $d=2$, Lemma \ref{mlemma} parts \eqref{tight} and \eqref{cov}
will allow us to use a weak law of large numbers to prove \eqref{done}.   Let
us note that the assumption that $\E |Y_0| < \infty$  is not necessary for
Lemma \ref{mlemma} part \eqref{cov}.

\begin{proof}[Proof of Proposition \ref{pertone} ($d=1$)]
Apply Lemma \ref{mlemma} \eqref{done} together with Proposition \ref{decayo}
\eqref{ch-decayb}.
\end{proof}

\begin{proof}[Proof of Lemma of \ref{mlemma} \eqref{tight}]

The following simple calculation (an instance of the `mass-transport
principle') shows that $\E \Lambda(0,1] =1$:
\[
\E \Lambda (0,1]
= \sum_{ z \in \Z} \P\big(Y_z +z \in (0,1]\big)
= \sum_{z \in \Z} \P\big(Y_0 \in (-z, -z+1]\big)= 1.
\]
Thus
\begin{equation}
\label{integervalued}
N_n =  \Lambda(-n, n] - 2n \ \text{for all} \  n \in \Z^{+}.
\end{equation}

For $A, B \subseteq \R$, write $$T_A ^B:= \# \ns{z \in A \cap \Z: z + Y_z \in
B};$$ that is, the number of $\Lambda$-points in $B$ that originated from
$A$.
Observe that for $n \in \Z^{+}$,
\begin{equation}
\label{fourterms}
N_n = T_{(n, \infty)} ^{(-n,n]} + T_{(-\infty, -n]} ^{(-n,n]}
- T_{(-n,n]} ^{(n, \infty)} - T_{(-n,n]} ^{(-\infty, -n]}.
\end{equation}
On the other hand, $\E |Y_0| < \infty$ implies easily that $K_+:= \E
T_{(-\infty, 0]} ^{[0, \infty)} < \infty$ and $K_-:= \E T_{[0, \infty)}
^{(-\infty, 0]} < \infty$.  By translation-invariance, each term on the right
side of \eqref{fourterms} is bounded in expectation by one of these
constants; for instance: $\E T_{(n, \infty)} ^{(-n,n]}\leq \E T_{[n, \infty)}
^{(-\infty,n]}=K_-$.  Hence $\E | N_n| \leq 2K_+ +2K_-$ for all $n\in
\Z^{+}$.
 \end{proof}

\begin{proof}[Proof of Lemma \ref{mlemma} \eqref{cov}]
Let $\goth{F}_n:=\sigma(\{z+Y_z \in [-n, n]\}: z \in \Z)$.   We will show
that for any event $E \in \sigma(Y_z:z\in\Z)$, we have
\begin{equation}
\label{asy}
\P(E \mid \goth{F}_n) \to \P(E) \  \text{a.s. as} \ n \to \infty.
\end{equation}
From \eqref{asy}, the result follows, since $\{N_n \leq \ell\}\in\goth{F}_n.$
It suffices to check \eqref{asy}  for $E$ in the generating algebra of events
that depend on only finitely many of the $Y_z$.  But for such an event, say
$E\in\sigma(Y_z:-m\leq z\leq m)$, we observe that $\P(E\mid\goth{F_n})$
equals the conditional probability of $E$ given the {\em finite}
$\sigma$-algebra $\sigma(\{z+Y_z\in[-n,n]\}: -m\leq z\leq m)$, hence the
required convergence follows from an elementary computation.
\end{proof}

\begin{proof}[Proof of Lemma \ref{mlemma} \eqref{done}]
By Lemma \ref{mlemma} \eqref{tight} we may choose an integer-valued $N$ and a
subsequence $(c_n)$ so that $N_{c_n} \stackrel{d}{\to} N$ as $n \to \infty$.
We will show that for all $\ell \in \Z$, there is a further subsequence
$c_{n_k} = :r_k$ such that
\begin{equation}
\label{almost}
\frac{1}{n}\sum_{k=1} ^n \Big[\mathbf{1}[N_{r_k} \leq \ell] - \P(N_{r_k} \leq  \ell) \Big]
 \  \stackrel{\P}{\to} \ 0  \ \text{as} \  n \to \infty.
\end{equation}
Clearly, the result follows from \eqref{almost} and the fact that $N_{r_k}
\stackrel{d}{\to} N$ as $k \to \infty$.

We use Corollary \ref{durrettm} in conjunction with a diagonal argument to
prove \eqref{almost}.   Consider an enumeration of the integers given by
$\ell_1, \ell_2, \ldots$ For each $i \in \Z^{+}$, let  $Z_{k} ^{i}:=
\mathbf{1}[N_{c_k} \leq \ell_i] - \P(N_{c_k} \leq \ell_i)$. By Lemma
\ref{mlemma} \eqref{cov} and Corollary \ref{durrettm}, there exists a
subsequence $c^{1}_{n_k} := r^1_k$ such that \eqref{almost} holds with $r_k$
replaced by $r^1_k$, and $\ell$ replaced by $\ell_1$.   Similarly, we may
choose $(r^2_k)$ to be a subsequence of $(r^1_k)$ so that  \eqref{almost}
holds with $r_k$ replaced by $r^2_k$, and $\ell$ replaced by $\ell_2$;
moreover Corollary \ref{durrettm} assures us that  \eqref{almost} holds with
$r_k$ replaced by $r^2_k$, and $\ell$ replaced by $\ell_1$.   Similarly
define the sequence $(r^{i}_k)$ for each $i \in \Z^{+}$.    By taking the
diagonal sequence $r_k := r^k_k$, we see that \eqref{almost} holds for all
$\ell \in \Z$.
\end{proof}

\subsection{Gaussian zeros in the hyperbolic plane}

The proof of Proposition~\ref{gausszeros} uses the following consequence  of
a result of Peres and Vir\'ag.

\begin{proposition}
\label{mini} If $\GZH$ is the Gaussian zero process on the hyperbolic plane
and $\GZH ^{*}$ is  its Palm version, then $\GZH^{*} \prec \GZH + \delta_0$
and $\GZH + \delta_0 \prec \GZH^{*}$.
\end{proposition}

\begin{proof}
Let $\GZH$ be the process of zeros of $\sum_{n=0} ^ {\infty} a_n z^n$, where
the $a_n$'s are i.i.d.\ standard complex Gaussian random variables. Let $E_k$
be the event that $\GZH (B(0,1/k )) >0$.  Peres and Vir\'ag
\cite[Lemma~18]{MR2231337} prove that the conditional law of
$(a_0,a_1,\ldots)$ given $E_k$ converges as $k\to\infty$ to the law of
$(0,\widehat{a}_1,a_2,\ldots)$, where $\widehat{a}_1$ is independent of the
$a_n$'s, and has a rotationally symmetric law with $|\widehat{a}_1|$ having
probability density $2r^3 e^{-r ^2}$.

Let $\bGZH$ be the process of zeros of the power series with coefficients
$(0,\widehat{a}_1,a_2,\ldots)$.  Since the latter sequence is mutually
absolutely continuous in law with $(0,a_1,a_2\ldots)$, we have that $\bGZH$
and $\GZH+\delta_0$ are mutually absolutely continuous in law.

By Rouch\'e's theorem from complex analysis \cite[Ch.\ 8, p.\ 229]{Gamelin},
the above convergence implies that the conditional law of $\GZH$ given $E_k$
converges to the law of $\bGZH$ (the convergence is in distribution with
respect to the vague topology for point processes).  By
\cite[Theorem~12.8]{MR818219} it follows that $\bGZH \eqd \GZH^{*}.$
\end{proof}

\begin{proof}[Proof of Proposition \ref{gausszeros}]
It follows from Proposition \ref{mini} and Theorems~\ref{thm-instol-stat-eq}
and \ref{suff} with Remark \ref{gen}    that the  Gaussian zero process on
the  hyperbolic plane is insertion-tolerant and deletion-tolerant.
\end{proof}

\section*{Acknowledgments}
We thank Omer Angel and Yuval Peres for many valuable conversations.   Terry
Soo thanks the organizers of the 2010 PIMS Summer School in Probability.

\bibliography{referencesu}
\bibliographystyle{abbrv}

\end{document}